\documentclass[11pt]{scrartcl}

\usepackage[utf8]{inputenc}
\usepackage[T1]{fontenc}
\usepackage{lmodern}
\usepackage{alphabeta}

\usepackage{titling}
\usepackage{xspace}

\usepackage[a4paper, top=3cm, bottom=2cm, left=2.5cm, right=2.5cm, includefoot]{geometry}

\usepackage[inline]{enumitem}

\usepackage{dsfont}
\usepackage{setspace}

\usepackage{bm}

\usepackage{amssymb}
\usepackage{microtype}
\usepackage{amsmath}
\usepackage{amssymb}
\usepackage{amsfonts}
\usepackage{mathtools}
\usepackage[mathscr]{euscript}
\usepackage{stmaryrd}
\SetSymbolFont{stmry}{bold}{U}{stmry}{m}{n}

\usepackage[hyphens]{url}
\usepackage{amsthm}
\usepackage{thmtools, thm-restate}

\usepackage{tikz}
\usepackage{xcolor}
\usepackage{subcaption}
\usepackage{pgf}

\usepackage{hyperref}
\usepackage{cleveref}

\usepackage{nicefrac}
\usepackage{verbatim}

\newcommand*{\abs}[1]{\ensuremath{{\lvert {#1} \rvert}}}

\usepackage{scalerel}

\makeatletter
\DeclareFontEncoding{LS1}{}{}
\makeatother
\DeclareFontSubstitution{LS1}{stix}{m}{n}
\DeclareSymbolFont{symbols2}{LS1}{stixfrak} {m} {n}

\DeclareMathDelimiter{\Lbrbrak}{\mathopen}{symbols2}{"22}{symbols2}{"22}
\DeclareMathDelimiter{\Rbrbrak}{\mathclose}{symbols2}{"23}{symbols2}{"23}

\newcommand{\maxP}[2]{
    \ifx\relax #2\relax
		\ensuremath{\mkern-2mu \Lbrbrak #1 \Rbrbrak \mkern-2mu}
	\else
		\ensuremath{\mkern-2mu \Lbrbrak #1 \Rbrbrak \mkern-2mu}(#2)
	\fi}

\colorlet{myGreen}{green!50!black}
\colorlet{myLightgreen}{green}
\colorlet{myRed}{red!90!black}
\definecolor{myBlue}{rgb}{0.25, 0.0, 1.0}
\definecolor{myLightBlue}{rgb}{0.39, 0.58, 0.93}
\colorlet{myViolet}{myBlue!55!myRed}
\definecolor{myOrange}{rgb}{1.0, 0.66, 0.07}

\definecolor{CornflowerBlue}{rgb}{0.39, 0.58, 0.93}
\definecolor{DarkGoldenrod}{rgb}{0.72, 0.53, 0.04}
\definecolor{BritishRacingGreen}{rgb}{0.0, 0.26, 0.15}
\definecolor{DarkMagenta}{rgb}{0.55, 0.0, 0.55}
\definecolor{AO}{rgb}{0.0, 0.5, 0.0}
\definecolor{BostonUniversityRed}{rgb}{0.8, 0.0, 0.0}
\definecolor{myRed}{rgb}{0.8, 0.0, 0.0}
\definecolor{DarkMidnightBlue}{rgb}{0.0, 0.2, 0.4}
\definecolor{DarkTangerine}{rgb}{1.0, 0.66, 0.07}
\definecolor{AppleGreen}{rgb}{0.55, 0.71, 0.0}
\definecolor{BrightUbe}{rgb}{0.82, 0.62, 0.91}
\definecolor{Amethyst}{rgb}{0.6, 0.4, 0.8}
\definecolor{DarkGray}{rgb}{0.52, 0.52, 0.51}
\definecolor{Gray}{rgb}{0.66, 0.66, 0.66}
\definecolor{BananaYellow}{rgb}{1.0, 0.88, 0.21}
\definecolor{Amber}{rgb}{1.0, 0.75, 0.0}
\definecolor{LightGray}{rgb}{0.83, 0.83, 0.83}
\definecolor{PrincetonOrange}{rgb}{1.0, 0.56, 0.0}
\definecolor{DeepCarrotOrange}{rgb}{0.91, 0.41, 0.17}
\definecolor{MidnightBlue}{rgb}{0.1, 0.1, 0.44}
\definecolor{HotMagenta}{rgb}{1.0, 0.11, 0.81}
\definecolor{Iceberg}{rgb}{0.44, 0.65, 0.82}
\definecolor{Byzantium}{rgb}{0.44, 0.16, 0.39}

\vfuzz \hfuzz
\setlist[itemize]{topsep=0pt,partopsep=0pt,itemsep=0pt,parsep=0pt}
\setlist[itemize,1]{label={\small\textbullet}}
\setlist[itemize,2]{label={\tiny\textbullet}}
\setlist[itemize,3]{label=$\cdot$}
\setlist[enumerate]{topsep=0pt,partopsep=0pt,itemsep=0pt,parsep=0pt}
\setlist[enumerate,1]{label=\roman*)}
\setlist[enumerate,2]{label=\alph*)}
\setlist[enumerate,3]{label=\arabic*)}

\hypersetup{
colorlinks=true,
linkcolor=AO!65!black,
citecolor=AO!65!black,
urlcolor=AppleGreen!65!black,
bookmarksopen=true,
bookmarksnumbered,
bookmarksopenlevel=2,
bookmarksdepth=3
}

\theoremstyle{definition}

\newtheorem{environment}{Environment}[section]

\newtheorem{lemma}[environment]{Lemma}
\crefname{lemma}{lemma}{lemmata}
\crefformat{lemma}{#2Lemma~#1#3}
\Crefformat{lemma}{#2Lemma~#1#3}

\newtheorem*{lemma*}{Lemma}
\crefname{lemma*}{lemma}{lemmata}
\crefformat{lemma*}{#2Lemma~#1#3}
\Crefformat{lemma*}{#2Lemma~#1#3}

\crefname{proposition}{proposition}{propositions}
\crefformat{proposition}{#2Proposition~#1#3}
\Crefformat{proposition}{#2Proposition~#1#3}

\newtheorem{corollary}[environment]{Corollary}
\crefname{corollary}{corollary}{corollaries}
\crefformat{corollary}{#2Corollary~#1#3}
\Crefformat{corollary}{#2Corollary~#1#3}

\newtheorem{theorem}[environment]{Theorem}
\crefname{theorem}{theorem}{Theorems}
\crefformat{theorem}{#2Theorem~#1#3}
\Crefformat{theorem}{#2Theorem~#1#3}

\crefname{conjecture}{conjecture}{Conjectures}
\crefformat{conjecture}{#2Conjecture~#1#3}
\Crefformat{conjecture}{#2Conjecture~#1#3}

\newtheorem*{hypothesis*}{Hypothesis}
\crefname{hypothesis*}{conjecture}{Conjectures}
\crefformat{hypothesis*}{#2Conjecture~#1#3}
\Crefformat{hypothesis*}{#2Conjecture~#1#3}

\newtheorem{observation}[environment]{Observation}
\crefname{observation}{observation}{Observations}
\crefformat{observation}{#2Observation~#1#3}
\Crefformat{observation}{#2Observation~#1#3}

\crefname{example}{example}{examples}
\crefformat{example}{#2Example~#1#3}
\Crefformat{example}{#2Example~#1#3}

\crefname{remark}{remark}{remarks}
\crefformat{remark}{#2Remark~#1#3}
\Crefformat{remark}{#2Remark~#1#3}

\crefname{figure}{figure}{figures}
\crefformat{figure}{#2Figure~#1#3}
\Crefformat{figure}{#2Figure~#1#3}

\crefname{equation}{equation}{Equations}
\crefformat{equation}{#2#1#3}
\Crefformat{equation}{#2#1#3}

\crefname{chapter}{chapter}{chapters}
\crefformat{chapter}{#2Chapter~#1#3}
\Crefformat{chapter}{#2Chapter~#1#3}

\crefname{section}{section}{sections}
\crefformat{section}{#2Section~#1#3}
\Crefformat{section}{#2Section~#1#3}

\crefname{algorithm}{algorithm}{algorithms}
\crefformat{algorithm}{#2Algorithm~#1#3}
\Crefformat{algorithm}{#2Algorithm~#1#3}

\crefname{notation}{notation}{notations}
\crefformat{notation}{#2Notation~#1#3}
\Crefformat{notation}{#2Notation~#1#3}

\crefname{question}{question}{questions}
\crefformat{question}{#2Question~#1#3}
\Crefformat{question}{#2Question~#1#3}

\crefname{problem}{problem}{problem}
\crefformat{problem}{#2problem~#1#3}
\Crefformat{problem}{#2problem~#1#3}

\crefname{claim}{claim}{claims}
\crefformat{claim}{#2Claim~#1#3}
\Crefformat{claim}{#2Claim~#1#3}

\crefname{definition}{definition}{definitions}
\crefformat{definition}{#2Definition~#1#3}
\Crefformat{definition}{#2Definition~#1#3}

\usetikzlibrary{calc}
\usetikzlibrary{fit}
\usetikzlibrary{decorations}
\usetikzlibrary{decorations.pathmorphing}
\usetikzlibrary{decorations.text}
\usetikzlibrary{external}
\usetikzlibrary{shapes,hobby}

\tikzset{
	position/.style args={#1:#2 from #3}{
		at=($(#3)+(#1:#2)$)
	}
}

\tikzset{
  v:main/.style = {draw, circle, scale=0.8, thick,fill=black,inner sep=0.7mm},
    v:maingray/.style = {draw, circle, scale=0.65, thick,color=gray,fill=gray,inner sep=0.7mm},
  v:mainempty/.style = {draw, circle, scale=0.8, thick,fill=white,inner sep=0.7mm},
  v:mainellipse/.style = {draw, ellipse, scale=0.8, thick,fill=white,inner sep=0.7mm},
    v:mainbox/.style = {draw, scale=0.8, thick,fill=white,inner sep=0.7mm},
    v:mainred/.style = {draw, circle, scale=0.65, thick,fill=red,inner sep=0.7mm},
        v:maingreen/.style = {draw, circle, scale=0.65, thick,fill=myGreen,inner sep=0.7mm},
                v:mainemptygreen/.style = {draw, circle, scale=0.65, thick,color=myGreen,fill=white,inner sep=0.7mm},
  v:mainemptygray/.style = {draw, circle, scale=0.65, thick,color=gray,fill=white,inner sep=0.7mm},
  v:tinytree/.style = {draw, circle, scale=0.03, thick,fill=black},
  v:middle/.style = {draw, circle, scale=0.3,thick,fill=Gray,color=Gray,inner sep=1mm},
  v:border/.style = {draw, circle, scale=0.75, thick,minimum size=10.5mm},
  v:mainfull/.style = {draw, circle, scale=1, thick,fill},
  v:ghost/.style = {inner sep=0pt,scale=1},
  v:marked/.style = {circle, scale=1.3, fill=DarkGoldenrod,opacity=0.4},
  v:tree/.style = {draw, circle, scale=0.45, thick,fill=black},
  >={latex},
  e:shiftedright/.style = {decoration={sl, raise=0.65pt},  decorate},
  e:shiftedleft/.style  = {decoration={sl, raise=-0.65pt}, decorate},
  e:marker/.style = {line width=8.5pt,line cap=round,opacity=0.35,color=DarkGoldenrod},
  e:colored/.style = {line width=1.8pt,color=BostonUniversityRed,cap=round,opacity=0.8},
  e:coloredthin/.style = {line width=1.6pt,opacity=1},
  e:coloredborder/.style = {line width=3.4pt},
  e:main/.style = {line width=1pt},
  e:thick/.style = {line width=2pt},
  e:mainplus/.style = {line width=1.15pt},
  e:mainthin/.style = {line width=0.6pt},
  e:extra/.style = {line width=1.3pt,color=LavenderGray},
  e:matching1/.style = {line width=2.2pt, color=myGreen, cap=round},
  e:matching2/.style = {line width=1.7pt, color=myRed, dashed, cap=round},
  e:matching2border/.style = {line width=0.35pt, color=white, dashed, cap=round, double=myRed, double distance=1.7pt},
  e:positive/.style = {line width=1.35pt},
  e:negative/.style = {line width=1.35pt,densely dotted},
}

\newcommand*{\N}{\mathds{N}}

\newcommand*{\Brace}[1]{\left( #1 \right)}
\newcommand*{\Fkt}[2]{
	\ifx\relax #2\relax
		#1
	\else
        #1(#2)
	\fi}

\DeclarePairedDelimiterX\Set[1]\{\}{%

#1
}

\newcommand*{\V}[1]{\Fkt{V}{#1}}

\newcommand*{\linegraph}[1]{\Fkt{\mathrm{I}}{#1}}
\renewcommand*{\deg}[1]{\Fkt{{\mathsf{deg}}}{#1}}
\renewcommand*{\log}{\mathsf{log}}

\newcommand{\InducedSubgraph}{\ensuremath{\subseteq_{\textnormal{i}}}}
\newcommand{\InducedMinor}{\ensuremath{\preccurlyeq_{\textnormal{i}}}}

\newcommand*{\rankwidth}[1]{\Fkt{{\mathsf{rw}}}{#1}}
\newcommand*{\tw}[1]{\Fkt{\mathsf{tw}}{#1}}

\newcommand*{\cdeg}[1]{\Fkt{{\mathsf{cdeg}}}{#1}}
\newcommand*{\cideg}[1]{\Fkt{{\Delta_{\mathcal{K}}}}{#1}}

\newcommand*{\bigPG}[3]{\Fkt{\mathsf{big}_{#1}^{#2}}{#3}} 
\newcommand*{\cutrank}[2]{\Fkt{\mathsf{cutrk}_{#1}}{#2}}
\newcommand*{\localcutrank}[2]{\Fkt{\mathsf{lcutrk}_{#1}}{#2}}

\newcommand*{\AM}[1]{\Fkt{\mathsf{Adj}}{#1}}
\newcommand*{\dn}[2]{\Fkt{\mathsf{dn}_{#1}}{#2}}
\newcommand*{\reduced}[2]{\Fkt{\widetilde{#1}}{#2}}
\newcommand*{\local}[2]{\Fkt{{#1}^{\mathsf{loc}}}{#2}}

\newcommand*{\ptw}[2]{\Fkt{{#1}\textnormal{-}\tw{}}{#2}}

\newcommand*{\ETierGraphs}[1]{\mathfrak{H}_d}
\newlength{\mylen}
\usepackage{wasysym}
\newcommand{\Star}{\mathscr{S}}
\newcommand{\matchingII}[0]{\ensuremath{\mathscr{M}^{\resizebox{6pt}{!}{$\Leftcircle\mkern-12mu\Rightcircle$}}}}
\newcommand{\matchingKI}[0]{\ensuremath{\mathscr{M}^{\resizebox{6pt}{!}{$\LEFTCIRCLE\mkern-12mu\Rightcircle$}}}}

\newcommand{\matchingKK}[0]{\ensuremath{\mathscr{M}^{\resizebox{6pt}{!}{$\LEFTCIRCLE\mkern-12mu\RIGHTCIRCLE$}}}}
\newcommand{\antimatchingII}[0]{\ensuremath{\mathscr{A}^{\resizebox{6pt}{!}{$\Leftcircle\mkern-12mu\Rightcircle$}}}}
\newcommand{\antimatchingKK}[0]{\ensuremath{\mathscr{A}^{\resizebox{6pt}{!}{$\LEFTCIRCLE\mkern-12mu\RIGHTCIRCLE$}}}}
\newcommand{\antimatchingKI}[0]{\ensuremath{\mathscr{A}^{\resizebox{6pt}{!}{$\LEFTCIRCLE\mkern-12mu\Rightcircle$}}}}
\newcommand{\halfgraphII}[0]{\ensuremath{\mathscr{H}^{\resizebox{6pt}{!}{$\Leftcircle\mkern-12mu\Rightcircle$}}}}
\newcommand{\halfgraphKK}[0]{\ensuremath{\mathscr{H}^{\resizebox{6pt}{!}{$\LEFTCIRCLE\mkern-12mu\RIGHTCIRCLE$}}}}
\newcommand{\halfgraphKI}[0]{\ensuremath{\mathscr{H}^{\resizebox{6pt}{!}{$\LEFTCIRCLE\mkern-12mu\Rightcircle$}}}}

\usepackage{tikzmacros}

\title{On graphs with a simple structure of maximal cliques}
\predate{}
\date{}
\postdate{}

\preauthor{}
\DeclareRobustCommand{\authorthing}{
	\begin{center}

        J.~Pascal Gollin\thanks{J.~Pascal Gollin was supported in part by the Institute for Basic Science (IBS-R029-Y3) and in part by the Slovenian Research and Innovation Agency (research project N1-0370).}\\ 
		{\small FAMNIT, University of Primorska, Koper, Slovenia.\\
	    \href{mailto:pascalgollin@ibs.re.kr}{pascal.gollin@famnit.upr.si}}\\
        \medskip

        Meike Hatzel\thanks{Meike Hatzel's research was supported by the Institute for Basic Science (IBS-R029-C1).}\\
		{\small Discrete Mathematics Group, Institute for Basic Science (IBS), Daejeon, South Korea.\\
	    \href{mailto:research@meikehatzel.com}{research@meikehatzel.com}}\\
		\medskip
	    
		\medskip
	    Sebastian Wiederrecht\\
		{\small KAIST, School of Computing, Daejeon, South Korea.\\
		\href{wiederrecht@kaist.ac.kr}{wiederrecht@kaist.ac.kr}}
\end{center}}
\author{\authorthing}
\postauthor{}

\begin{document}
\maketitle
\thispagestyle{empty}

\begin{abstract}
We say that a hereditary graph class $\mathcal{G}$ is \emph{clique-sparse} if there is a constant $k=k(\mathcal{G})$ such that for every graph $G\in\mathcal{G}$, every vertex of $G$ belongs to at most $k$ maximal cliques, and any maximal clique of $G$ can be intersected in at most $k$ different ways by other maximal cliques.

We provide various characterisations of clique-sparse graph classes including a list of five parametric forbidden induced subgraphs.
We show that recent techniques for proving induced analogues of Menger's Theorem and the Grid Theorem of Robertson and Seymour can be lifted to prove induced variants in clique-sparse graph classes when replacing ``treewidth'' by ''tree-independence number''.
\end{abstract}

\section{Introduction}\label{sec:Intro}

At the heart of the graph minors theory of Robertson and Seymour \cite{RobertsonS1986-GM5,RobertsonS2003GM16} sits a structural distinction between those graphs that can be obtained by being ``glued together'' along small vertex sets from graphs of bounded \textsl{size} and those graphs that can be obtained in the same way from graphs of arbitrary size but with restricted structure.
Key to the distinction between these two cases is the celebrated \emph{Grid Theorem} \cite{RobertsonS1986-GM5} stating that a minor-closed graph class falls into the first category, witnessed by small \textsl{treewidth}, if and only if it excludes a planar graph.

In recent years, there has been a push to understand this kind of structural behaviour through the lens of induced graph containment relation such as \emph{induced subgraphs} (i.e.~subgraphs obtained only from the deletion of vertices) and \emph{induced minors} (i.e.~minors obtained only from the deletion of vertices and contraction of edges). 
One of the main goals in this area of study is finding useful analogues for the Grid Theorem, and more fundamentally Menger's Theorem, in the induced settings.

Both were accomplished, in their most desirable form, in the setting of graphs of bounded degree.
An induced variant for Menger's Theorem (see \cref{subsec:clique-sparseGraphs}) for graphs with constant maximum degree was independently discovered by Gartland, Korhonen, and Lokshtanov~\cite{GartlandKL2023-inducedMengerMaxDeg} as well as by Hendrey, Norin, Steiner, and Turcotte~\cite{HendreyNST2023-inducedMengerMaxDeg}.
Moreover, Albrechtsen, Huynh, Jacobs, Knappe, and Wollan~\cite{AlbrechtsenHTJKW2024-inducedMenger2paths} and independently Georgakopoulos and Papasoglu \cite{GeorgakopoulosP2023} conjectured that instead measuring the \textsl{size} of a minimum separator, a different measure, namely the \textsl{domination number} of~$G[S]$, could result in a meaningful generalisation of Menger's theorem for the induced setting.

In searching for an induced version of the Grid Theorem, large cliques are an obstacle to a verbatim analogue in the induced setting. 
Korhonen~\cite{Korhonen2023-inducedgridmaxdegree} proved that finding a ${(k \times k)}$-grid as an induced minor is possible when the treewidth is at least~$f(k,\Delta(G))$ for some function~${f(k,d) \in \mathcal{O}(k^{10} + 2^{d^5})}$. 
But as before, replacing treewidth with a different measure could yield a more meaningful generalisation beyond bounded maximum degree. 
For example, Dallard, Krnc, Kwon, Milani\v{c}, Munaro, \v{S}torgel, and Wiederrecht~\cite{DallardKKMMW2024} conjectured that replacing treewidth by $\alpha$-treewidth\footnote{Also known as \emph{tree-independence number}. The parameter was independently introduced by Yolov \cite{Yolov2018MinorMatching} and Dallard, Milanič, and Štorgel \cite{dallard2021TreewidthVSClique1,Dallard2024TreewidthVSClique2}.} (where we replace the measure of the size of a largest bag in a tree-decomposition by its independence number, denoted by~$\ptw{\alpha}{G}$, see \cref{sec:Prelim} for more details) is enough for graphs that do not contain large stars as induced subgraphs.
Notice that, when using the maximum size of an independent set contained in a vertex set $S$ as the measure of $S$, the two suggested settings: Excluding an induced star and measuring the set in terms of its domination number, become interchangeable.

\paragraph{Our results.}
We say that the maximum number of leaves of an induced star in a graph $G$ is the \emph{local independence number}, denoted by $\local{\alpha}{G}$, of $G$.

This paper aims to propose a sensible interpolation between graphs of bounded maximum degree and graphs of bounded local independence number.

Bounding the maximum degree of a graph~$G$ puts bounds on both the size of every maximal clique of~$G$ and the size of the neighbourhoods of such cliques.
In particular, this puts bounds on the following two aspects of maximal cliques:
\begin{enumerate}
    \item The number of different maximal cliques a vertex may belong to, and
    \item the number of different ways a single maximal clique may be intersected by other maximal cliques in~$G$.
\end{enumerate}
Let us call the first number the \emph{clique-incidence-degree} and denote it by~$\cideg{G}$, 
and let us call the second number the \emph{clique-incidence-diversity} and denote it by~$\reduced{\omega}{G}$.
A graph class $\mathcal{G}$ is said to be hereditary if for every $G\in\mathcal{G}$, every induced subgraph of $G$ also belongs to $\mathcal{G}$.
We say that a hereditary graph class is \emph{clique-sparse} if there exists a constant $k$ such that both the clique-incidence-degree and the clique-diversity-degree of every graph in $\mathcal{G}$ is at most $k$.

We provide several characterisations for clique-sparse graph classes (see \cref{thm:clique-sparse}), including a characterisation in terms of forbidden induced subgraphs.
Moreover, we show that clique-sparse graphs, while more general than graphs of bounded degree, still allow the techniques of Gartland et al., Hendrey et al., and Korhonen \cite{GartlandKL2023-inducedMengerMaxDeg,HendreyNST2023-inducedMengerMaxDeg,Korhonen2023-inducedgridmaxdegree} to obtain induced variants of Menger's Theorem and the Grid Theorem for the setting of $\alpha$-treewidth.

Our last result is concerned with the connection between $\alpha$-treewidth and rankwidth (see \cref{sec:rankwidth} for the definition). 
For every graph~$G$, the $\alpha$-treewidth of~$G$ is bounded in a function of the local independence number and the rankwidth. 
In general, the converse is not true. 
However, in \cref{sec:rankwidth} we prove that the converse does hold for clique-sparse graphs. 
Moreover, we show that if we consider the classes of graphs obtained by excluding some, but not all, of the forbidden induced subgraphs for clique-sparse graphs mentioned above, then the reverse does not hold (unless the class of graphs is a class of cographs). 
This yields further understanding of clique-sparse graphs from a different angle.

In the following, we discuss the results of Gartland et al.\@, Hendrey et al.\@, and Korhonen \cite{GartlandKL2023-inducedMengerMaxDeg,HendreyNST2023-inducedMengerMaxDeg,Korhonen2023-inducedgridmaxdegree} in slightly more detail and provide a more in-depth discussion of our results.

\subsection{Clique-sparse graph classes}\label{subsec:clique-sparseGraphs}

Two of the most fundamental tools of structural graph theory are Menger's Theorem and the Grid Theorem of Robertson and Seymour. 
Menger's Theorem states that for every positive integer~$k$ and two sets~$A$ and~$B$ of vertices of graph~$G$, the graph either contains an \emph{$A$-$B$-linkage of order~$k$} (that is, a set~$\mathcal{P}$ of~$k$ vertex-disjoint paths whose union~${\bigcup \mathcal{P}}$ is a subgraph of~$G$), or an \emph{$A$-$B$-separator of order~$k$} (that is, a set~$S$ of~$k$ vertices for which~$G - S$ has no component containing vertices from both~$A$ and~$B$). 
The Grid Theorem states that there exist a function~$f$ such that for every positive integer~$k$, every graph contains either the ${(k \times k)}$-grid as a minor or has treewidth\footnote{We give a formal definition of treewidth in \cref{sec:Prelim}.} at most~$f(k)$. 

A generalisation of Menger's theorem to the \textsl{induced setting} wants to find an \emph{induced $A$-$B$-linkage of order~$k$} (that is, a set~$\mathcal{P}$ of~$k$ vertex-disjoint paths whose union~${\bigcup \mathcal{P}}$ is an induced subgraph of~$G$). 
Large cliques offer an obstacle for a verbatim approach of a generalisation: Between disjoint sets~$A$ and~$B$ of vertices in a large clique there is neither an induced $A$-$B$-linkage of order~$2$ nor an $A$-$B$-separator of small order.
The results of Gartland et al.~\cite{GartlandKL2023-inducedMengerMaxDeg} and Hendrey et al.~\cite{HendreyNST2023-inducedMengerMaxDeg} offer approximate versions of such a theorem, where there either is an induced $A$-$B$-linkage of order~$k$ or an $A$-$B$-separator whose order is bounded by~$f(k,\Delta(G))$ for a function ${f(k,d) \in \mathcal{O}(k(d+1)^{d^2 + 1})}$.

\paragraph{Clique-incidence-diversity and clique-incidence-degree: lifting the induced variant of Menger's Theorem.}

Let us first precisely state the version of the theorem of Gartland et al.

\begin{theorem}[Gartland, Korhonen, Lokshtanov 2023$^{+}$~\cite{GartlandKL2023-inducedMengerMaxDeg}]
    \label{thm:gartlandinducedmenger}
    Let $\Delta$ be a positive integer.
    For every graph $G$ of maximum degree $\Delta$, every two sets $A,B\subseteq \V{G}$, and every positive integer $k$, there either exists an induced $A$-$B$-linkage of order $k$ in $G$, or a set $S$ of size at most $k\cdot (\Delta+1)^{\Delta^2+1}$ vertices such that there does not exist an $A$-$B$-path in $G-X$.
\end{theorem}

The strategies of the two groups Gartland et al.~\cite{GartlandKL2023-inducedMengerMaxDeg} and Hendrey et al.~\cite{HendreyNST2023-inducedMengerMaxDeg} to prove their induced variants of Menger's theorem are very similar, they just have slightly different perspectives. 
For the sake of simplicity, we explain the strategy of Gartland et al.~\cite{GartlandKL2023-inducedMengerMaxDeg}. 
Their argument is based on the idea of starting with a large general flow and reducing it to a smaller, but still large, induced flow.
It splits into two main steps.
In the first step, they consider a set $X$ of vertices such that the members in~$X$ are pairwise at distance at least three to each other.
This means that the neighbourhoods of any two members of~$X$ do not intersect.
Let~$G_X$ be obtained from~$G$ by contracting the closed neighbourhood of every vertex in~$X$.
As the maximum degree of~$G$ is bounded, $N_G[x]$ contains only a bounded number of vertices for every vertex $x$.
So, of a large flow between two sets in~$G$, there still must remain a somewhat (depending on the maximum degree) large flow between the remainders of these sets in the contracted graph.

As the second step, they observe that on graphs with maximum degree~$\Delta$, one may colour the vertices with at most~$\Delta^2+1$ many colours such that any two vertices with the same colour are in distance at least three to each other.
Thus, the strategy established in the first step can be applied to each colour class separately, and because there are only a bounded number of colour classes, there still remains a large flow.

By controlling the two graph parameters \textsl{clique-incidence-degree} and \textsl{clique-incidence-diversity} directly, we obtain a strict generalisation of the regime of graphs with bounded maximum degree while maintaining the validity of \cref{thm:gartlandinducedmenger} (and later on also \cref{thm:maxdegreegridtheorem}) up to small adjustments. 

Of course, simply bounding these parameters would again not allow for a verbatim generalisation, since, as before, a large clique would be a counterexample. 
But, as hinted at in the introduction, a different measure of~$S$ yields the solution. 
Instead of the domination number~$\gamma(G[S])$ as suggested by Albrechtsen et al.~\cite{AlbrechtsenHTJKW2024-inducedMenger2paths} and Georgakopoulos and Papasoglu \cite{GeorgakopoulosP2023}, in this setting, we can improve the result by using the \emph{clique-cover number}~$\theta(G[S])$, that is the smallest number of cliques in~$G[S]$ for which~$S$ is covered by the vertex-sets of these cliques, as the measure for the separator. 
It is straightforward to see that~${\gamma(G) \leq \theta(G)}$ (since the set obtained by picking a vertex from every clique of a smallest clique-cover is a dominating set).
Hence, if we find a separator~$S$ for which~$\theta(G[S])$ is bounded, so is~$\gamma(G[S])$.

With these notions, we can prove the following theorem by reducing it to the result of Gartland et al.

\begin{restatable}{theorem}{hedgehogmenger}
    \label{thm:hedgehogmenger}
    Let $s,t$ be positive integers. 
    For every graph~$G$ with ${\reduced{\omega}{G} \leq s}$ and~${\cideg{G} \leq t}$, every two sets~${A,B \subseteq V(G)}$, and every positive integer~$k$, 
    there either exists an induced $A$-$B$-linkage~$\mathcal{L}$ of order~$k$ in~$G$, or a set~$S$ with~${\theta(G[S]) \leq k \cdot (t(s-1)+1)^{t^2(s-1)^2+1}}$ such that there exists no $A$-$B$-path in~${G-S}$. 
\end{restatable}

\paragraph{Clique-incidence-diversity and local independence number: lifting the induced variant of the Grid Theorem.}

In a similar fashion, Korhonen \cite{Korhonen2023-inducedgridmaxdegree} has proved a version of the Grid Theorem for induced minors in graphs with bounded maximum degree. 

\begin{theorem}[Korhonen, 2023 \cite{Korhonen2023-inducedgridmaxdegree}]
    \label{thm:maxdegreegridtheorem}
    There exists a function $f(k,\Delta)\in\mathcal{O}(k^{10}+2^{\Delta^5})$ such that for every positive integer $k$ and every graph $G$ with $\Delta(G)\leq \Delta$ it holds that, if $\mathsf{tw}(G)>f(k,\Delta)$, then $G$ contains the $(k\times k)$-grid as an induced minor.
\end{theorem}

The proof of this theorem does not directly use a contraction trick similar to the one above.
However, considering vertices with pairwise far enough distance together with the fact that any distance-$r$-neighbourhood in graphs with small degrees must be small plays a huge role here as well.
While we again require the \textsl{clique-incidence-diversity} of our graphs to be bounded, the second parameter we bound this time is the \textsl{local independence number} rather than the clique-incidence-degree.

Our aim here is to make progress on the aforementioned conjecture of Dallard et al.~\cite{DallardKKMMW2024} by generalising Korhonen's result similarly as we did with Menger's theorem.
As before, we cannot use the size of the bags of a tree-decomposition as our measure, but in our setting, we can use the largest clique-cover number of a bag instead. 
We call the resulting parameter the $\theta$-treewidth (denoted by $\ptw{\theta}{G}$). 
We give a precise definition in \cref{sec:Prelim} and observe in \cref{sec:parametercomp} that ${\ptw{\alpha}{G} \leq \ptw{\theta}{G}}$. 

As before, we prove the following theorem by reducing it to the result of Korhonen. 

\begin{restatable}{theorem}{hedgehoggrid}
    \label{thm:hedgehoggrid}
    There exists a function ${f(k,s,d) \in \mathcal{O}(k^{10} + 2^{d^{5s}})}$ such that for every positive integer~$k$ and every graph $G$
    with~${\reduced{\omega}{G} \leq s}$ 
    and~${\local{\alpha}{G} \leq d}$ holds that, if ${\ptw{\theta}{G} > f(k,s,d)}$, then~$G$ contains the ${(k \times k)}$-grid as an induced minor. 
\end{restatable}

\paragraph{Clique-sparse graphs.}
At first glance, the class of graphs for which we prove an approximate induced Menger's Theorem with \cref{thm:hedgehogmenger} and the class of graphs for which we prove an induced Grid Theorem with \cref{thm:hedgehoggrid} may be different, but as it turns out, they are essentially the same. 
In fact, we see in \cref{sec:parametercomp} that in classes of graphs with bounded clique-incidence diversity~($\reduced{\omega}{}$) the clique-incidence degree~($\cideg{}$) is bounded if and only if the local independence number~($\local{\alpha}{}$) is bounded. 

The way we reduce \cref{thm:hedgehogmenger} to \cref{thm:gartlandinducedmenger} and \cref{thm:hedgehoggrid} to \cref{thm:maxdegreegridtheorem} actually reveals the strong and simple connection these graphs have to graphs of bounded maximum degree:
For a hereditary graph class $\mathcal{G}$, having bounded clique-incidence-degree and bounded clique-incidence-diversity is equivalent to the property that every graph $G\in\mathcal{G}$ can be turned into a graph of bounded degree by iteratively identifying \textsl{true twins}.
Here \emph{true twins} are two vertices with the same closed neighbourhood. 
Despite the rather simple structure of this graph class, its equivalence to the aforementioned classes, as well as multiple other equivalences that we discuss below, reveals a new depth to these graphs.

In \cref{sec:parametercomp} we also see that in graph classes where~$\reduced{\omega}{}$ is bounded, the \emph{local clique-cover number}~$\local{\theta}{}$, that is, the maximum (taken over all vertices~$v$) of the clique-cover number of the closed neighbourhood of $v$, is bounded if and only if~$\local{\alpha}{}$ is bounded. 
Hence, the class of graphs~$G$ where the maximum of~$\reduced{\omega}{G}$ and~$\local{\theta}{G}$ is bounded is clique-sparse.

\begin{figure}
    \colorlet{K4colour}{DeepCarrotOrange}
    \colorlet{K3colour}{AppleGreen}
    \colorlet{K2colour}{CornflowerBlue}
    \colorlet{K1colour}{purple}
    \centering
    \begin{tikzpicture}
        \node (centre) at (0,0) {};

        \node (G) at ($(centre)+(90:1)$) {
            \begin{tikzpicture}
                \node (G-centre) at (0,0) {};

                \node[vertex,myRed] (x-1) at ($(G-centre)$) {};
                \node (x-label) at ($(x-1)+(115:0.3)$) {\textcolor{myRed}{$x$}};
                \node (K4-centre) at ($(x-1)+(200:1)$) {};
                \node[vertex] (x-2) at ($(K4-centre)+({360/5+20}:1)$) {};
                \node[vertex] (x-3) at ($(K4-centre)+({360/5*2+20}:1)$) {};
                \node[vertex] (x-4) at ($(K4-centre)+({360/5*3+20}:1)$) {};
                \node[vertex] (x-5) at ($(K4-centre)+({360/5*4+20}:1)$) {};
                \node[vertex] (x-6) at ($(x-1)+(340:1)$) {};
                \node[vertex] (x-7) at ($(x-1)+(45:1)$) {};
                \node[vertex] (x-8) at ($(x-5)+(240:1)$) {};
                \node[vertex] (x-9) at ($(x-5)+(330:1)$) {};
                \node[vertex] (x-10) at ($(x-7)+(135:1)$) {};

                \draw[edge,K4colour] (x-1) to (x-2);
                \draw[edge,K4colour] (x-1) to (x-3);
                \draw[edge,K4colour] (x-1) to (x-4);
                \draw[edge,K4colour] (x-1) to (x-5);
                \draw[edge,K4colour] (x-2) to (x-3);
                \draw[edge,K4colour] (x-2) to (x-4);
                \draw[edge,K4colour] (x-2) to (x-5);
                \draw[edge,K4colour] (x-3) to (x-4);
                \draw[edge,K4colour] (x-3) to (x-5);
                \draw[edge,K4colour] (x-4) to (x-5);

                \begin{pgfonlayer}{background}
                    \filldraw[K4colour!30!white] (x-1.center) -- (x-2.center)
                        -- (x-3.center) -- (x-4.center) -- (x-5.center);
                \end{pgfonlayer}

                \node (K4-label) at ($(K4-centre) + (180:1.5)$) {\textcolor{K4colour}{$K_4$}};

                \draw[edge,K3colour] (x-1) to (x-6);
                \draw[edge,K3colour] (x-1) to (x-7);
                \draw[edge,K3colour] (x-5) to (x-6);
                \draw[edge,K3colour] (x-5) to (x-7);
                \draw[edge,K3colour] (x-6) to (x-7);

                \begin{pgfonlayer}{background}
                    \filldraw[K3colour!30!white] (x-1.center) -- (x-5.center)
                        -- (x-6.center) -- (x-7.center);
                \end{pgfonlayer}

                \node (K3-label) at ($(x-6) + (45:0.5)$) {\textcolor{K3colour}{$K_3$}};

                \draw[edge,K2colour] (x-2) to (x-7);
                \draw[edge,K2colour] (x-10) to (x-7);
                \draw[edge,K2colour] (x-2) to (x-10);
                \draw[edge,K2colour] (x-1) to (x-10);

                \begin{pgfonlayer}{background}
                    \filldraw[K2colour!30!white] (x-1.center) -- (x-2.center) -- (x-10.center)
                        -- (x-7.center);
                \end{pgfonlayer}

                \node (K2-label) at ($(x-10) + (350:0.8)$) {\textcolor{K2colour}{$K_2$}};

                \draw[edge,K1colour] (x-5) to (x-8);
                \draw[edge,K1colour] (x-5) to (x-9);
                \draw[edge,K1colour] (x-8) to (x-9);

                \begin{pgfonlayer}{background}
                    \filldraw[K1colour!30!white] (x-5.center) -- (x-8.center)
                        -- (x-9.center);
                \end{pgfonlayer}

                \node (K1-label) at ($(x-5) + (285:1.1)$) {\textcolor{K1colour}{$K_1$}};
                
            \end{tikzpicture}
        };

        \node (G-label) at ($(G)+(110:1.8)$) {$G$};

        \node (linegraph) at ($(centre)+(200:4)$) {
            \begin{tikzpicture}
                \node (linegraph-centre) at (0,0) {};

                \node[vertex,K1colour,scale=1.7] (K-1-v) at ($(linegraph-centre) + (270:1)$) {};
                \node[vertex,K2colour,scale=1.7] (K-2-v) at ($(linegraph-centre) + (90:1)$) {};
                \node[vertex,K3colour,scale=1.7] (K-3-v) at ($(linegraph-centre) + (0:1)$) {};
                \node[vertex,K4colour,scale=1.7] (K-4-v) at ($(linegraph-centre) + (180:1)$) {};

                \draw (K-1-v) to (K-3-v);
                \draw (K-1-v) to (K-4-v);
                \draw (K-2-v) to (K-3-v);
                \draw (K-2-v) to (K-4-v);
                \draw (K-3-v) to (K-4-v);
            \end{tikzpicture}
        };

        \node (quotient) at ($(centre)+(340:4)$) {
            \begin{tikzpicture}
                \node (Q-centre) at (0,0) {};

                \node[vertex] (x-1) at ($(Q-centre)$) {};
                \node (K4-centre) at ($(x-1)+(200:1)$) {};
                \node[vertex] (x-2) at ($(K4-centre)+({360/5+20}:1)$) {};
                \node[] (x-3) at ($(K4-centre)+({360/5*2+20}:1)$) {};
                \node[] (x-4) at ($(K4-centre)+({360/5*3+20}:1)$) {};
                \node[vertex,K4colour] (x-3-4) at ($(x-3)!0.5!(x-4)$) {};
                \node[vertex] (x-5) at ($(K4-centre)+({360/5*4+20}:1)$) {};
                \node[vertex] (x-6) at ($(x-1)+(340:1)$) {};
                \node[vertex] (x-7) at ($(x-1)+(45:1)$) {};
                \node[] (x-8) at ($(x-5)+(240:1)$) {};
                \node[] (x-9) at ($(x-5)+(330:1)$) {};
                \node[vertex,K1colour] (x-8-9) at ($(x-8)!0.5!(x-9)$) {};
                \node[vertex] (x-10) at ($(x-7)+(135:1)$) {};

                \draw[edge,K4colour] (x-1) to (x-2);
                \draw[edge,K4colour] (x-1) to (x-3-4);
                \draw[edge,K4colour] (x-1) to (x-5);
                \draw[edge,K4colour] (x-2) to (x-3-4);
                \draw[edge,K4colour] (x-2) to (x-5);
                \draw[edge,K4colour] (x-3-4) to (x-5);

                \begin{pgfonlayer}{background}
                    \filldraw[K4colour!30!white] (x-1.center) -- (x-2.center)
                        -- (x-3-4.center) -- (x-5.center);
                \end{pgfonlayer}

                \draw[edge,K3colour] (x-1) to (x-6);
                \draw[edge,K3colour] (x-1) to (x-7);
                \draw[edge,K3colour] (x-5) to (x-6);
                \draw[edge,K3colour] (x-5) to (x-7);
                \draw[edge,K3colour] (x-6) to (x-7);

                \begin{pgfonlayer}{background}
                    \filldraw[K3colour!30!white] (x-1.center) -- (x-5.center)
                        -- (x-6.center) -- (x-7.center);
                \end{pgfonlayer}

                \draw[edge,K2colour] (x-2) to (x-7);
                \draw[edge,K2colour] (x-10) to (x-7);
                \draw[edge,K2colour] (x-2) to (x-10);
                \draw[edge,K2colour] (x-1) to (x-10);

                \begin{pgfonlayer}{background}
                    \filldraw[K2colour!30!white] (x-1.center) -- (x-2.center) -- (x-10.center)
                        -- (x-7.center);
                \end{pgfonlayer}

                \draw[edge,K1colour] (x-5) to (x-8-9);
                
                \begin{pgfonlayer}{background}
                    \filldraw[K1colour!30!white]  ($(x-5)+(185:0.04)$)
                    to[bend right] ($(x-8-9)+(185:0.04)$) to ($(x-8-9)+(5:0.04)$) to[bend right] ($(x-5)+(5:0.05)$);
                \end{pgfonlayer}
            \end{tikzpicture}
        };

        \draw[->,-latex,shorten >=15mm,shorten <=15mm] (G.center) -- (quotient.center) node[pos=0.6,below=1mm] {$\widetilde{G}$};
        \draw[->,-latex,shorten >=15mm,shorten <=15mm] (G.center) -- (linegraph.center) node[pos=0.6,below=1mm] {$\linegraph{\mathcal{K}(G)}$};
    \end{tikzpicture}
    
    \caption{The graph $G$ has four maximal cliques $\mathcal{K}(G) = \Set{\textcolor{K1colour}{K_1},\textcolor{K2colour}{K_2},\textcolor{K3colour}{K_3},\textcolor{K4colour}{K_4}}$.
        On the left side of the illustration, we see the linegraph $\linegraph{\mathcal{K}(G)}$.
        On the right we see the quotient graph $\widetilde{G}$ of $G$.
        The coloured vertices in $\widetilde{G}$ represent contracted classes of true twins.
        Note that $\cideg{G}=3$, $\reduced{\omega}{G} = 4$, and $\cdeg{G} = 3$.}
    \label{fig:illustrate_different_parameters}
\end{figure}

We make use of another graph parameter that is equivalent to bounding both~$\reduced{\omega}{G}$ and~$\cideg{G}$ simultaneously: the \emph{clique-degree}, denoted by~$\cdeg{G}$, which we call the maximum number of distinct maximal cliques that a maximal clique intersects, see \cref{fig:illustrate_different_parameters} for a comparison of these three parameters.

Perhaps the most powerful characterisation of these parameters being bounded is in terms of forbidden induced subgraphs. 
In \cref{sec:forbiddeninducedsubgraphs}, we see that forbidding a star (\cref{fig:matchingKI_and_matchingKK}(a)), a split matching (\cref{fig:matchingKI_and_matchingKK}(b)), a matching between two cliques (\cref{fig:matchingKI_and_matchingKK}(c)), an anti-matching between two cliques (\cref{fig:matchingKI_and_matchingKK}(d)), and a half-graph between two cliques (\cref{fig:matchingKI_and_matchingKK}(e)) as an induced subgraph is equivalent to being clique-sparse. 

\begin{figure}[!ht]
    \centering
    \includegraphics[scale=0.8]{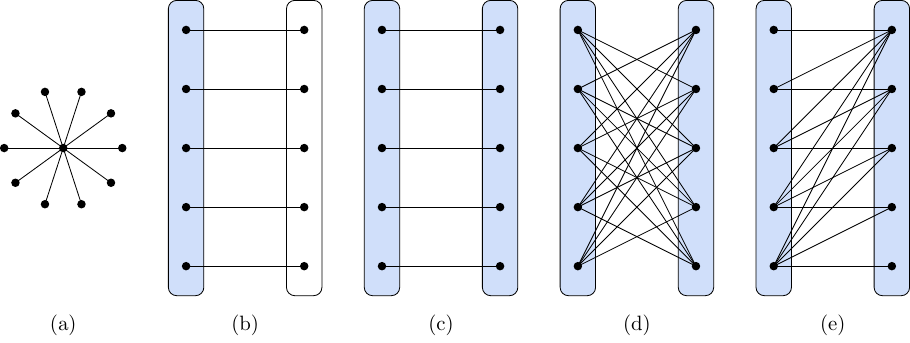}
    \caption{The forbidden induced subgraphs for clique-sparse graphs: 
        the \textcolor{CornflowerBlue}{filled} ellipses represent cliques, while the empty ellipses represent independent sets. See \cref{sec:Prelim} for precise definitions. }
    \label{fig:matchingKI_and_matchingKK}
\end{figure}

Let us summarise these characterisations of clique-sparseness in the following theorem. 

\begin{theorem}
    \label{thm:clique-sparse}
    Let~$\mathcal{G}$ be a hereditary class of graphs. 
    Then the following statements are equivalent. 
    \begin{enumerate}
        [label=(\arabic*)]
        \item $\mathcal{G}$ is clique-sparse (i.e.~both $\reduced{\omega}{}$ and $\cideg{}$ are bounded). 
        \item There is a constant $d$ such that for every graph in~$G$, the graph obtained from successively identifying true twins in~$G$ has maximum degree~$d$. 
        \item Both $\reduced{\omega}{}$ and $\local{\alpha}{}$ are bounded in $\mathcal{G}$.
        \item Both $\reduced{\omega}{}$ and $\local{\theta}{}$ are bounded in $\mathcal{G}$.
        \item $\cdeg{}$ is bounded in $\mathcal{G}$.
        \item $\mathcal{G}$ excludes a star, a split matching, a matching between two cliques, an anti-matching between two cliques, and a half-graph between two cliques as induced subgraph. 
    \end{enumerate}
\end{theorem}

As mentioned earlier, our last result establishes a close relation between $\alpha$-treewidth and rankwidth in clique-sparse graph classes. 
In particular, we establish in \cref{sec:rankwidth} exactly which of the following: a split matching, a matching between two cliques, an anti-matching between two cliques, a half-graph between two cliques, and a star, a graph class needs to exclude in order for $\alpha$-treewidth to be equivalent to the maximum of the local independence number and the rankwidth.
Indeed, this is the case if and only if the class of graphs excludes~$P_4$ and a star, or it excludes one of each of the aforementioned graph families.

\section{Preliminaries}
\label{sec:Prelim}

Given two integers $a,b\in\mathbb{Z},$ we denote by $[a,b]$ the set $\{ x\in\mathbb{Z} \mid a\leq x\leq b \}.$
Notice that $[a,b]$ is empty in case $b<a.$
Moreover, for a single integer $k\in\mathbb{Z},$ we denote by $[k]$ the set $[1,k].$
Unless stated otherwise, all logarithms in this paper are assumed to be binary. 

Let $G$ and $H$ be graphs and $e=xy\in E(G)$ be an edge.
We say that $H$ is obtained from $G$ by \emph{contracting} $e$ if there exists a vertex $u\in V(H)$ such that $H-u\equiv G-x-y$ and $N_H(u)= N_G(x)\cup N_G(y)$.
If $H$ is obtained from $G$ by repeatedly contracting edges, we say that $H$ is \emph{obtained from $G$ by edge contractions}.

We say that a graph $H$ is an \emph{induced minor} of a graph $G$ if $H$ can be obtained from $G$ by vertex deletions and edge contractions.

Let $G$ be a graph and $X\subseteq V(G)$.
We denote by $N_G(X)$ the \emph{neighbourhood} of $X$ in $G$, that is the set ${\{ v\in V(G)\setminus X \mid vx\in E(G)\text{ for some }x\in X \}}$. 
Moreover, $N_G[X]\coloneqq N_G(X)\cup X$ is called the \emph{closed neighbourhood} of $X$.
For a positive integer $r$, the \emph{distance-$r$-neighbourhood} of a vertex $v\in V(G)$, denoted by $N^r_G(v)$, is the set $\{ u\in V(G)\setminus\{ v\} \mid \mathsf{dist}_G(u,v)\leq r \}$.
The \emph{closed distance-$r$-neighbourhood} of $v$, denoted by $N_G^r[v]$, is the set $N_G^r(v)\cup\{ v\}$.

Let $G$ be a graph and $r\in \N$ be a positive integer.
The \emph{$r$-th power} of $G$, denoted by $G^r$, is the graph with vertex set $V(G)$ where any two distinct vertices $u,v\in V(G)$ are adjacent if and only if $\mathsf{dist}_G(u,v)\leq r$.

Let $G$ be a graph.
In what follows, we often work on the following hypergraph structure induced by the cliques of $G$.
Let $\mathcal{K}(G)$ denote the set of all vertex sets of $G$ which induce maximal cliques in $G$.
Then $(V(G),\mathcal{K}(G))$ forms a hypergraph.
As a convention, we suppress isolated vertices in hypergraphs, and thus, it suffices to denote a hypergraph as its edge set.
In particular, we can consider $\mathcal{K}(G)$ as a hypergraph. 
By slight abuse of notation, we may refer to~${K \in \mathcal{K}(G)}$ to be the maximal clique of $G$ instead of writing~$G[K]$. 

Let $\mathcal{E}$ be a hypergraph.
The \emph{linegraph} of $\mathcal{E}$, denoted by $\linegraph{\mathcal{E}}$, is the graph with vertex set $\mathcal{E}$ where two vertices $e,f\in\mathcal{E}$ are adjacent if and only if they are distinct and $e\cap f\neq \emptyset$. 

Let $k$ be a positive integer, $G$ be a graph, and $A,B \subseteq V(G)$ be sets of vertices.
An \emph{$A$-$B$-linkage} of \emph{order $k$} is a subgraph $L$ of $G$ consisting of $k$ pairwise vertex-disjoint $A$-$B$-paths.
An $A$-$B$-linkage is \emph{induced} if $L=G[V(L)]$. 

For a set~$X$, let~$K(X)$ denote the clique with vertex set~$X$. 
For two graphs~$G$ and~$H$, we define the \emph{union}~$G \cup H$ as the graph on~$V(G) \cup V(H)$ with edge set~$E(G) \cup E(H)$. 

\paragraph{Graph parameters and parametric graphs.}
A \emph{graph parameter} is a map $\mathsf{p}$ that assigns to every graph $G$ a non-negative integer value $\mathsf{p}(G)$.
Two graph parameters $\mathsf{p}$ and $\mathsf{q}$ are said to be \emph{asymptotically equivalent}, or simply \emph{equivalent}, if there exist functions $f,g\colon \N\to\N$ such that for all graphs $G$ it holds that
\begin{align*}
f(\mathsf{p}(G))\leq \mathsf{q}(G)\leq g(\mathsf{p}(G)).
\end{align*}
Let $\preceq$ be a graph containment relation.
A graph parameter $\mathsf{p}$ is said to be \emph{$\preceq$-monotone} if for every pair of graphs $H\preceq G$ it holds that $\mathsf{p}(H)\leq \mathsf{p}(G)$.

In this paper, we are mainly concerned with two graph containment relations.
The first is the \emph{induced subgraph relation}, denoted by $\InducedSubgraph$.
A graph $H$ is an \emph{induced subgraph} of a graph $G$, denoted~$H\InducedSubgraph G$, if it can be obtained from $G$ by deleting vertices.
A graph parameter $\mathsf{p}$ is said to be \emph{hereditary} if it is $\InducedSubgraph$-monotone.

The second relation is the \emph{induced minor relation}, denoted by $\InducedMinor$.
A graph $H$ is an \emph{induced minor} of a graph $G$, denoted~$H\InducedMinor G$, if it can be obtained from $G$ by a sequence of vertex deletions and edge contractions.
We say that a graph parameter $\mathsf{p}$ is \emph{induced minor monotone} if it is $\InducedMinor$-monotone.

Let $\mathsf{p}_1,\dots,\mathsf{p}_{\ell}$ be graph parameters.
We denote by $\maxP{\mathsf{p}_1,\dots,\mathsf{p}_{\ell}}{}$ the parameter defined as $\maxP{\mathsf{p}_1,\dots,\mathsf{p}_{\ell}}{G} \coloneqq \max \{ \mathsf{p}_1(G),\dots,\mathsf{p}_{\ell}(G)\}$.

For any graph parameter~$\mu$, we define a \emph{local} version of~$\mu$ via 
        \[
            \local{\mu}{G} \coloneqq \max \{ \mu(N[v]) \mid v \in V(G) \}.
        \]

\medskip

A \emph{parametric graph} is a sequence $\mathscr{H}=\langle \mathscr{H}_i\rangle_{i\in\N}$.
Given a graph containment relation $\preceq$ we say that a parametric graph $\langle \mathscr{H}_i\rangle_{i\in\N}$ is \emph{$\preceq$-monotone} if $\mathscr{H}_i\preceq \mathscr{H}_{i+1}$ for all $i\in\N$. 

Let~$\mathfrak{F}=\{ \langle\mathscr{H}^1_j\rangle_{j\in\N},\dots,\langle\mathscr{H}^{\ell}_j\rangle_{j\in\N}\}$ be a set of parametric graphs and let~$\preceq$ be a graph containment relation. 
We define the following graph parameter
\begin{align*}
\bigPG{\mathfrak{F}}{\preceq}{G}\coloneqq \max\{ k \mid \text{there exists }i\in[\ell]\text{ such that }\mathscr{H}^i_k\preceq G \}.
\end{align*}
In case $\mathfrak{F}$ contains a single parametric graph $\mathscr{H}$ we write $\bigPG{\mathscr{H}}{\preceq}{G}$ instead of $\bigPG{\mathfrak{F}}{\preceq}{G}$.

\medskip

Now, let us discuss some specific parametric graphs.

Let~$X$ and~$Y$ be sets. 
We denote by~$K(X)$ the complete graph with vertex set~$X$ and by~$K(X,Y)$ the complete bipartite graph with bipartition classes~$X$ and~$Y$.
Now suppose~${X = \{ x_1, \dots, x_n \}}$ and~$Y = \{ y_1, \dots, y_n \}$ are ordered sets of the same size.
Then
\begin{itemize}
    \item let $\mathscr{M}(X,Y)$ denote the bipartite graph with vertex set~$X \cup Y$ and edge set~$\{ x_i y_i \mid i \in [n] \}$, which we call the \emph{matching} between~$X$ and~$Y$,
    \item let~$\mathscr{A}(X,Y)$ bipartite graph with vertex set~$X \cup Y$ and edge set~$\{ x_i y_j \mid i,j \in [n], \, i \neq j \}$, which we call the \emph{anti-matching} between~$X$ and~$Y$, and
    \item let $\mathscr{H}(X,Y)$ bipartite graph with vertex set~$X \cup Y$ and edge set~$\{ x_i y_j \mid i,j \in [n], \, i \leq j \}$, which we call the \emph{half-graph} between~$X$ and~$Y$. 
\end{itemize}

We use these notations to denote several parametric graphs, see also \cref{fig:allhedgehogs}. 
For each positive integer~$i$, let~${X_i = \{ x_1, \dots, x_i \}}$ and~$Y_i = \{ y_1, \dots, y_i \}$ denote (ordered) sets of the same size, and let~${c \notin X_i \cup Y_i}$

\begin{itemize}
    \item Let~${\Star_{i} = K(\{c\},X_i)}$ denote the \emph{star} with with~$i$ leaves, and let 
    $\Star = \langle {\Star_{i}} \rangle_{i \in \N}$ 
    denote the parametric graph of \emph{stars}.
    
    \item Let~$\matchingII_i = \mathscr{M}(X_i,Y_i)$ denote the \emph{matching} of order~$i$, and let 
    $\matchingII = \langle \matchingII_i \rangle_{i \in \N}$ 
    denote the parametric graph of \emph{matchings}.
    
    \item Let~$\matchingKI_i = \mathscr{M}(X_i,Y_i) \cup K(X_i)$ denote the \emph{split matching} of order~$i$, and let 
    $\matchingKI = \langle \matchingKI_i \rangle_{i \in \N}$ 
    denote the parametric graph of \emph{split matchings}.
    
    \item Let~$\matchingKK_i = \mathscr{M}(X_i,Y_i) \cup K(X_i) \cup K(Y_i)$ denote the \emph{clique matching} of order~$i$, and let 
    $\matchingKK = \langle \matchingKK_i \rangle_{i \in \N}$ 
    denote the parametric graph of \emph{clique matchings}.
    
    \item Let~$\antimatchingII_i = \mathscr{A}(X_i,Y_i)$ denote the \emph{anti-matching} of order~$i$, and let 
    $\antimatchingII = \langle \antimatchingII_i \rangle_{i \in \N}$ 
    denote the parametric graph of \emph{anti-matchings}.
    
    \item Let~$\antimatchingKI_i = \mathscr{A}(X_i,Y_i) \cup K(X_i)$ denote the \emph{split anti-matching} of order~$i$, and let 
    $\antimatchingKI = \langle \antimatchingKI_i \rangle_{i \in \N}$ 
    denote the parametric graph of \emph{split anti-matchings}.
    
    \item Let~$\antimatchingKK_i = \mathscr{A}(X_i,Y_i) \cup K(X_i) \cup K(Y_i)$ denote the \emph{clique anti-matching} of order~$i$, and let 
    $\antimatchingKK = \langle \antimatchingKK_i \rangle_{i \in \N}$ 
    denote the parametric graph of \emph{clique anti-matchings}.
    
    \item Let~$\halfgraphII_i = \mathscr{H}(X_i,Y_i)$ denote the \emph{half-graph} of order~$i$, and let 
    $\halfgraphII = \langle \halfgraphII_i \rangle_{i \in \N}$ 
    denote the parametric graph of \emph{half-graphs}.
    
    \item Let~$\halfgraphKI_i = \mathscr{H}(X_i,Y_i) \cup K(X_i)$ denote the \emph{split half-graphs} of order~$i$, and let 
    $\halfgraphKI = \langle \halfgraphKI_i \rangle_{i \in \N}$ 
    denote the parametric graph of \emph{split half-graphs}.
    
    \item Let~$\halfgraphKK_i = \mathscr{H}(X_i,Y_i) \cup K(X_i) \cup K(Y_i)$ denote the \emph{clique half-graphs} of order~$i$, and let 
    $\halfgraphKK = \langle \halfgraphKK_i \rangle_{i \in \N}$ 
    denote the parametric graph of \emph{clique half-graphs}. 
\end{itemize}

\begin{figure}[!ht]
    \centering
    \includegraphics[scale=0.5]{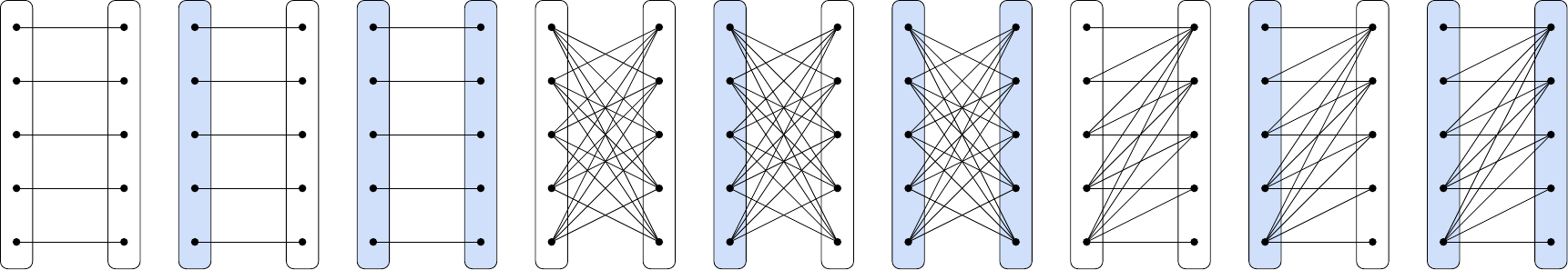}
    \caption{From left to right, the graphs~$\matchingII_5$, $\matchingKI_5$, $\matchingKK_5$, $\antimatchingII_5$, $\antimatchingKI_5$, $\antimatchingKK_5$, $\halfgraphII_5$, $\halfgraphKI_5$, $\halfgraphKK_5$. 
        The \textcolor{CornflowerBlue}{filled} ellipses represent cliques, while the empty ellipses represent independent sets. }
    \label{fig:allhedgehogs}
\end{figure}

Let $\mathcal{C}$ be a graph class and $\preceq$ be a graph containment relation.
We say that $\mathcal{C}$ is \emph{closed under $\preceq$} if for every pair $H\preceq G$ of graphs $G\in\mathcal{C}$ implies that $H\in\mathcal{C}$.
A graph class is \emph{hereditary} if it is closed under $\InducedSubgraph$ and \emph{closed under induced minors} if it is $\InducedMinor$-closed.

\paragraph{Graph measures and tree-decompositions.}

A graph measure is a function $\mu$ that, for every graph $G$, assigns to each vertex set $X\subseteq V(G)$ a non-negative integer $\mu_G(X)$.
For the two graph parameters $\alpha$ and $\theta$ we define corresponding graph measures by $\mu_G(X)\coloneqq \mu(G[X])$ for both $\mu\in\{ \alpha,\theta\}$.
Generally, if the graph is understood from the context, we drop the subscript in this notation.
In this paper,r we are mostly concerned with the graph measures $\alpha$, $\theta$, and the \textit{cardinality}, denoted by $\abs{{}\cdot{}}$.

A \emph{tree-decomposition} for a graph $G$ is a tuple $\mathcal{T}=(T,\beta)$ where $T$ is a tree and $\beta\colon V(T)\to 2^{V(G)}$, called the \emph{bags} of $\mathcal{T}$, such that
\begin{enumerate}
	\item $\bigcup_{t\in V(T)}\beta(t)=V(G)$,
	\item for every $e\in E(G)$ there exists some $t\in V(T)$ such that $e\subseteq \beta(t)$, and
	\item for every vertex $v\in V(G)$ the set $\{ t\in V(T) \mid v\in\beta(T) \}$ induces a subtree of $T$.
\end{enumerate}
Given a graph measure $\mu$, the \emph{$\mu$-width} of $\mathcal{T}$ is defined as the value~${\max_{t\in V(T)}\mu_G(\beta(t))}$ and the $\mu$-treewidth of~$G$, denoted by $\mu$-$\mathsf{tw}(G)$, is the minimum $\mu$-width over all tree-decompositions for~$G$.

For $\mu=\abs{{}\cdot{}}$, the definition above yields a parameter which is (up to a $-1$ in the original definition) equal to \emph{treewidth}.
For $\mu=\alpha$ we obtain the definition of $\alpha$-treewidth, or sometimes called the \textit{tree-independence-number}.
Finally, for $\theta$ we obtain the notion of \emph{clique-cover-treewidth}.
In slight abuse of notation
we write $\tw{G}$ instead of $\abs{{}\cdot{}}\text{-}\tw{G}$. 

Let~$G$ be a graph and~$\mu$ be a graph measure.
A \emph{separation} in~$G$ is a tuple~${(A,B)}$ such that there is no edge between~${A \setminus B}$ and~${B \setminus A}$.
The \emph{$\mu$-order} of ${(A,B)}$ is defined to be the value~${\mu(A \cap B)}$.

\section{Clique-incidence-diversity, clique-incidence-degree, and the clique-quotient graph}
\label{sec:cid}

Let~$G$ be a graph. 
Recall that~$\mathcal{K}(G)$ denotes the set of all vertex sets of~$G$ which induce maximal cliques in~$G$.
We define an equivalence relation~$\sim_{\mathcal{K}(G)}$ on~$V(G)$ by saying that~${v \sim_{\mathcal{K}(G)} w}$ if and only if for all~${K \in \mathcal{K}(G)}$, 
\[{v \in K} \textnormal{ if and only if } {w \in K}.\]
We denote by~$\widetilde{V(G)}$ the set of equivalence classes of~$\sim_{\mathcal{K}(G)}$ and by~$\widetilde{G}$ the graph on~$\widetilde{V(G)}$ obtained from~$G$ by identifying equivalent vertices. 
We call~${\widetilde{G}}$ the \emph{clique-quotient graph} of~$G$. 
Observe that two vertices~${v,w}$ of~$G$ are equivalent, if and only if they are \emph{true twins} in~$G$, that is, for every~${x \in V(G)}$, we have that~${vx \in E(G)}$ if and only if~${vy \in E(G)}$. 
In other words, the equivalence classes are precisely the \emph{modules} of~$G$ (that are, the sets~$X$ of vertices such that each vertex outside of~$X$ is either adjacent to each vertex in~$X$ or no vertex in~$X$) that are also cliques, and the clique-quotient graph is obtained from~$G$ by contracting every such clique-module. 
Throughout this paper, for ease of notation, for each~${v \in V(G)}$, we denote by~$\widetilde{v}$ the equivalence class of~$v$ with respect to~$\sim_{\mathcal{K}(G)}$, and for~${X \subseteq V(G)}$, we denote by~${\widetilde{X}}$ the set~$X /{\sim_{\mathcal{K}(G)}} = \{ \widetilde{x} \in \widetilde{V(G)} \mid x \in X \}$. 
On the other hand, we may introduce an arbitrary vertex~${\widetilde{v} \in V(\widetilde{G})}$ without having introduced~$v$ beforehand. 
In this case, we may implicitly talk about~${v}$ as a representative of~$\widetilde{v}$. 
Similarly, we may introduce an arbitrary subset~$\widetilde{X} \subseteq V(\widetilde{G})$ without having introduced a set~$X$ beforehand. 
We often make use of the set~${\bigcup \widetilde{X} \coloneqq \{ x \in V(G) \mid \widetilde{x} \in \widetilde{X} \}}$.
Note that~${\widetilde{\bigcup \widetilde{X}} = \widetilde{X}}$. 
Additionally, when we introduce~$\widetilde{X} \subseteq V(\widetilde{G})$ for a set~$X$ of representatives, then we take~$X$ to be a set that, for each~$\widetilde{x} \in \widetilde{X}$, contains exactly one element of~$\widetilde{x}$. 

Let~${x,y \in V(G)}$ with~${x \not\sim_{\mathcal{K}(G)} y}$. 
By definition of~$\widetilde{G}$, we have
\begin{equation}
    {xy \in E(G)} \,\text{ if and only if }\, \widetilde{x}\widetilde{y} \in E(\widetilde{G}). \tag{$\ast$} \label{eq:ast}
\end{equation}
Expanding on this idea, we observe the following. 

\begin{observation}
    \label{obs:GtildeEdges}
    Let~$G$ be a graph, let~${X,Y \subseteq V(G)}$. 
    Then the following properties hold. 
    \begin{enumerate}
        [label=(\arabic*)]
        \item \label{obs:GtildEedges:1} There is an edge between~$\widetilde{X}$ and~$\widetilde{Y}$ in~$E(\widetilde{G})$ if and only if there is an edge between~$\bigcup \widetilde{X}$ and~$\bigcup \widetilde{Y}$ in~$E(G)$. 
        \item \label{obs:GtildEedges:2} If there is an edge between~$\widetilde{X}$ and~$\widetilde{Y}$ in~$E(\widetilde{G})$, then there is an edge between~$X$ and~$Y$ in~$E(G)$. 
        \item \label{obs:GtildEedges:3} If there is an edge between~$X$ and~$Y$ in~$E(G)$, then either there is an edge between~$\widetilde{X}$ and~$\widetilde{Y}$ in~$E(\widetilde{G})$, or a vertex in~$\widetilde{X} \cap \widetilde{Y}$.
    \end{enumerate}
\end{observation}

\begin{proof}
    Let~${x \in X}$ and~${y \in Y}$ with~${x \not\sim_{\mathcal{K}(G)} y}$. 
    By~(\cref{eq:ast}), ${xy \in E(G)}$ if and only if~${\widetilde{x}\widetilde{y} \in E(\widetilde{G})}$. 
    And since~${x \in \bigcup \widetilde{X}}$ and~${y \in \bigcup \widetilde{Y}}$, \ref{obs:GtildEedges:1} follows. 
    For~\ref{obs:GtildEedges:2}, if~${\widetilde{x}\widetilde{y} \in E(\widetilde{G})}$ for some~${\widetilde{x} \in \widetilde{X}}$ and~${\widetilde{y} \in \widetilde{Y}}$, then by definition there is an~${x' \in X \cap \widetilde{x}}$ and~${y' \in Y \cap \widetilde{Y}}$, and it follows that~${xy \in E(G)}$. 
    Finally, if~${xy \in E(G)}$ for some~${x \in X}$ and~${y \in Y}$, then either~${x \not \sim_{\mathcal{K}(G)} y}$ and hence~${\widetilde{x}\widetilde{y} \in E(\widetilde{G})}$ by~(\cref{eq:ast}), or~${x \sim_{\mathcal{K}(G)} y}$ and hence~$\widetilde{x} \in \widetilde{X} \cap \widetilde{Y}$, proving \ref{obs:GtildEedges:3}.
\end{proof}

For a maximal clique~${K \in \mathcal{K}(G)}$, we call the number of equivalence classes of~$\sim_{\mathcal{K}(G)}$ that intersect~$K$ the \emph{incidence-diversity} of~$K$. 
Then, the \emph{clique-incidence-diversity} of~$G$ is the maximum incidence-diversity over all~${K \in \mathcal{K}(G)}$, which we denote by~$\reduced{\omega}{G}$. 

For a vertex~${v \in V(G)}$, we call the number of maximal cliques $K\in\mathcal{K}(G)$ with $v\in K$ the \emph{clique-incidence-degree} of~$v$.
The \emph{clique-incidence-degree} of~$G$ is the maximum clique-incidence-degree over all~${v \in V(G)}$, which we denote by~$\cideg{G}$. 
Note that~$\cideg{G}$ is equal to the maximum degree of the hypergraph~${\mathcal{K}(G)}$. 

A key observation on the behaviour of~$\sim_{\mathcal{K}(G)}$ is that taking its quotient preserves a lot of the original structure of~$G$ from the perspective of maximal cliques.
We illustrate this behaviour with the following observations.

\begin{lemma}
    \label{lem:clique-quotient-graph}
    For every graph~$G$, the map~${\widetilde{\pi} \colon V(G) \to V(\widetilde{G})}$ that maps every vertex~$v$ of~$G$ to its equivalence class~$\widetilde{v}$ with respect to~$\sim_{\mathcal{K}(G)}$ induces a bijection~$\kappa$ between~${\mathcal{K}(G)}$ and~${\mathcal{K}(\widetilde{G})}$. 
    
    In particular, 
    \begin{enumerate}
        [label=(\arabic*)]
        \item \label{item:lem:clique-quotient-graph1} $\kappa$ is an isomorphism between~$\linegraph{\mathcal{K}(G)}$ and~$\linegraph{\mathcal{K}(\widetilde{G})}$, 
    \end{enumerate}
    and, for every set~$R$ of representatives of the equivalence classes of~$\sim_{\mathcal{K}(G)}$, 
    \begin{enumerate}
        [label=(\arabic*), resume]
        \item \label{item:lem:clique-quotient-graph2} the restriction of~$\widetilde{\pi}$ to~$R$ is an isomorphism between the hypergraph ${\mathcal{K}(\widetilde{G})}$ and the subhypergraph of~$\mathcal{K}(G)$ induced by~$R$, and
        \item \label{item:lem:clique-quotient-graph3} the restriction of~$\widetilde{\pi}$ to~$R$ is an isomorphism between $\widetilde{G}$ and~$G[R]$. 
    \end{enumerate}
\end{lemma}

\begin{proof}
    By (\cref{eq:ast}), ${\widetilde{\pi}(K) = \widetilde{K}}$
    induces a maximal clique of~$\widetilde{G}$ for every~${K \in \mathcal{K}(G)}$. 
    On the other hand, $\widetilde{x}$ induces a clique for each~${x \in V(G)}$. 
    So for each~${\widetilde{K} \in \mathcal{K}(\widetilde{G})}$, the union~$\bigcup \widetilde{K}$ is a clique of~$G$. 
    Straightforwardly, the maximality of~$\widetilde{G}[\widetilde{K}]$ implies the maximality of~$G[\bigcup \widetilde{K}]$. 

    Clearly, for~${K, K' \in \mathcal{K}(G)}$ we observe that~${K \cap K' = \emptyset}$ if and only if~${\widetilde{K} \cap \widetilde{K'} = \emptyset}$, and hence~$\linegraph{\mathcal{K}(G)}$ is isomorphic to~$\linegraph{\mathcal{K}(\widetilde{G})}$. 

    Lastly, observe that if~$R$ is a set of representatives of the equivalence classes with respect to~$\sim_{\mathcal{K}(G)}$, then the restriction of~$\widetilde{\pi}$ to~$R$ is a bijection, and hence, by the reasoning above, an isomporphism between the subhypergraph of~$\mathcal{K}(G)$ induced by~$R$ and~$\mathcal{K}(\widetilde{G})$, as well as~$\widetilde{G}$ and~$G[R]$. 
\end{proof}

A straightforward consequence of this lemma is that the clique-incidence-diversity~$\reduced{\omega}{G}$ of a graph~$G$ is equal to the clique number~$\omega(\widetilde{G})$ of~$\widetilde{G}$.
Indeed, it follows that several graph parameters on~$G$ and~$\widetilde{G}$ are equal, as we summarise in the following corollary.

\begin{corollary}
    \label{cor:Gtildeproperties}
    Let~$G$ be a graph.
    \begin{enumerate}
        [label=(\arabic*)]
        \item \label{item:cor:Gtildeproperties1} ${\reduced{\omega}{G} = \omega(\widetilde{G})}$. 
        \item \label{item:cor:Gtildeproperties2} ${\cideg{G} = \cideg{\widetilde{G}}}$.
        \item \label{item:cor:Gtildeproperties3} ${\alpha(G) = \alpha(\widetilde{G})}$ 
            and~${\local{\alpha}{G} = \local{\alpha}{\widetilde{G}}}$.
        \item \label{item:cor:Gtildeproperties5} ${\theta(G) = \theta(\widetilde{G})}$ 
            and~${\local{\theta}{G} = \local{\theta}{\widetilde{G}}}$.
    \end{enumerate}
\end{corollary}

\begin{proof}
    The bijection~$\kappa$ from \cref{lem:clique-quotient-graph} maps a maximal clique~${K \in \mathcal{K}(G)}$ to~${\widetilde{K}}$. 
    Since the latter has exactly as many vertices as~$K$ contains equivalence classes of~$\sim_{\mathcal{K}(G)}$, it follows that~${\reduced{\omega}{G} = \omega(\widetilde{G})}$. 

    Property~\ref{item:cor:Gtildeproperties2} follows immediately from \cref{lem:clique-quotient-graph}\ref{item:lem:clique-quotient-graph1} together with the fact that the maximum degree of a hypergraph is equal to the clique number of its linegraph. 

    Since every independent set of~$G$ contains at most one vertex from each equivalent class of~$\sim_{\mathcal{K}(G)}$, it follows from \cref{lem:clique-quotient-graph}\ref{item:lem:clique-quotient-graph3} that~${\alpha(G) = \alpha(\widetilde{G})}$. 
    Similarly, as we may assume 
    a clique-cover to consist only of maximal cliques, it follows from \cref{lem:clique-quotient-graph}\ref{item:lem:clique-quotient-graph1} that~${\theta(G) = \theta(\widetilde{G})}$. 
    From these observations, it immediately follows that the local variants of both of these parameters are also equal on~$G$ and~$\widetilde{G}$. 
\end{proof}

For any graph parameter~$\mu$, let~$\reduced{\mu}{}$ be the graph parameter defined by~${\reduced{\mu}{G} \coloneqq \mu(\widetilde{G})}$. 
In particular, let us turn our attention to the maximum degree of the clique-quotient graph, that is, the graph parameter~$\reduced{\Delta}{}$, and its interaction with both the clique-incidence-diversity and the clique-incidence-degree.

\begin{lemma}
    \label{lem:cidcidegdelta}
    For every graph~$G$, 
    \begin{enumerate}
        [label=(\arabic*)]
        \item ${\reduced{\omega}{G} \leq \reduced{\Delta}{G}}$, unless~$\widetilde{G}$ is edgeless (in which case~${\reduced{\omega}{G} \leq 1}$ and~${\reduced{\Delta}{G} = 0}$), 
        \item ${\cideg{G} \leq 3^{\lceil{\reduced{\Delta}{G}/3}\rceil}}$, and
        \item ${\reduced{\Delta}{G} \leq \cideg{G} \cdot (\reduced{\omega}{G} - 1)}$. 
    \end{enumerate}
\end{lemma}

\begin{proof}
    Clearly, ${\reduced{\omega}{G} \leq \reduced{\Delta}{G} + 1}$. 
    Suppose~$G$ is a graph for which~${k \coloneqq \reduced{\omega}{G} = \reduced{\Delta}{G} + 1}$. 
    In particular, $\widetilde{G}$ contains a clique~$\widetilde{K}$ of size~$k$. 
    Since~${\Delta(\widetilde{G}) = k - 1}$, each vertex in~$\widetilde{K}$ has degree~${k-1}$ in~$\widetilde{G}$, and hence~${\widetilde{G}[\widetilde{K}]}$ is a component of~$\widetilde{G}$. 
    But since no other maximal clique of~$G$ intersects~${K \coloneqq \bigcup\widetilde{K}}$, all vertices in~$K$ are equivalent, and hence~${k = 1}$. 
    It follows that~${\Delta(\widetilde{G}) = 0}$, and hence~$\widetilde{G}$ is edgeless. 
    
    To show that $\cideg{G} \leq 3^{\lceil{\reduced{\Delta}{G}}/3\rceil}$, first recall that~$\cideg{G} = \cideg{\widetilde{G}}$, see \cref{cor:Gtildeproperties}. 
    Moreover, for every vertex~${\widetilde{v} \in V(\widetilde{G})}$, every~${\widetilde{K} \in \mathcal{K}(\widetilde{G})}$ that contains~$\widetilde{v}$ is contained in~$N_{\widetilde{G}}[\widetilde{v}]$. 
    In particular, the number of maximal cliques that contain~$\widetilde{v}$ is equal to the number of maximal cliques of~${G[N_{\widetilde{G}}(\widetilde{v})]}$. 
    Using the bound of Moon and Moser~\cite{MoonM1965cliques} that an~$n$-vertex graph contains at most~$3^{\lceil n/3 \rceil}$ maximal cliques, we obtain the desired bound. 

    For the third inequality, consider a vertex~${\widetilde{v} \in V(\widetilde{G})}$. 
    By \cref{lem:clique-quotient-graph}, $\widetilde{v}$ is adjacent to exactly those~$\widetilde{w}$ for which there is a maximal clique~${K \in \mathcal{K}(G)}$ with~${\widetilde{v} \cup \widetilde{w} \subseteq K}$. 
    And since there are at most~$\cideg{G}$ such maximal cliques, each of which contains at most~$(\reduced{\omega}{G} - 1)$ equivalence classes that are distinct from~$\widetilde{v}$, the result follows. 
\end{proof}

First, let us observe that all these inequalities are tight. 
For the first inequality, consider~$\matchingKI_n$, where~${\reduced{\omega}{\matchingKI_n} = \reduced{\Delta}{\matchingKI_n} = n}$. 
For the second inequality, consider the graph~$G$ obtained from a complete multipartite graph with~$n$ partition classes, each of size~$3$, by adding a universal vertex. 
Clearly, this graph has~$3^n$ maximal cliques (cf.~Moon and Moser~\cite{MoonM1965cliques}) which, furthermore, distinguish all vertices with respect to~$\sim_{\mathcal{K}(G)}$. 
Hence~${G = \widetilde{G}}$ and hence~${\reduced{\Delta}{G} = \Delta(G) = 3n}$. 
For the third inequality, observe that~${\reduced{\Delta}{\Star_n} = \Delta(\Star_n) = n = \cideg{\Star_n}}$ and that~${\reduced{\omega}{\Star_n} = \omega(\Star_n) = 2}$.

The first two inequalities from the previous lemma imply that~${\maxP{\cideg{},\reduced{\omega}{}}{G} \leq 3^{\lceil{\reduced{\Delta}{G}/3}\rceil}+1}$ for all graphs~$G$.
So in particular, the class of all graphs with~${\maxP{\cideg{},\reduced{\omega}{}}{}}$ at most~${3^{\lceil{d/3}\rceil}+1}$ is strictly more general than the class of graph of maximum degree at most~$d$. 
Moreover, $\matchingKI_n$ is a graph with~${\cideg{\matchingKI_n} = 2}$ and~${\reduced{\Delta}{\matchingKI_n} = \reduced{\omega}{\matchingKI_n} = n}$, while~$\Star_n$ is a graph with~${\reduced{\omega}{\Star_n} = 2}$ and~${\reduced{\Delta}{\Star_n} = \cideg{\Star_n} = n}$. 
Hence, bounding either~$\cideg{}$ or~$\reduced{\omega}{}$ while not bounding the other leads to a class strictly less general than bounding~$\reduced{\Delta}{}$. 
From the first inequality of the previous lemma, we observe that~$\reduced{\Delta}{}$ and~$\maxP{\reduced{\Delta}{},\reduced{\omega}{}}{}$ are equivalent. 
Together with the third inequality, we observe that~$\reduced{\Delta}{}$ is equivalent to~$\maxP{\cideg{},\reduced{\omega}{}}{}$. 

\begin{corollary}
    \label{cor:cid-cideg-Deltatilde}
    The parameters~$\reduced{\Delta}{}$ and~$\maxP{\cideg{},\reduced{\omega}{}}{}$ are equivalent. 
    \qed
\end{corollary}

\section{An approximate Menger theorem and an induced grid theorem on graphs with bounded clique-incidence-diversity and bounded clique-incidence-degree}

Let us first prove the generalisation of \cref{thm:gartlandinducedmenger} to graphs of bounded $\maxP{\cideg{},\reduced{\omega}{}}{}$. 
To that end, we first investigate how separations and paths relate between~$G$ and~$\widetilde{G}$ in the following two lemmas. 

\begin{lemma}
    \label{lemma:quotientseparations}
    Let~$k$ be a positive integer.
    For any graph~$G$, 
    \begin{enumerate}
        [label=(\arabic*)]
        \item if $(A,B)$ is a separation of~$G$ of $\theta$-order~$k$, then~$(\widetilde{A}, \widetilde{B})$ is a separation of~$\widetilde{G}$ of order at most~$k \cdot \reduced{\omega}{G}$, and
        \item if $(\widetilde{A},\widetilde{B})$ is a separation in~$\widetilde{G}$ of order~$k$, then $(\bigcup\widetilde{A},\bigcup\widetilde{B})$ is a separation in $G$ of $\theta$-order at most~$k$. 
    \end{enumerate}
\end{lemma}

\begin{proof}
    By \cref{obs:GtildeEdges}\ref{obs:GtildEedges:3} and since~$\widetilde{A} \setminus \widetilde{B}$ and~$\widetilde{B} \setminus \widetilde{A}$ are disjoint, if~$(A,B)$ is a separation of~$G$, then~$(\widetilde{A}, \widetilde{B})$ is a separation of~$\widetilde{G}$. 
    More straight-forwardly, by \cref{obs:GtildeEdges}\ref{obs:GtildEedges:1}, if~$(\widetilde{A},\widetilde{B})$ is a separation in~$\widetilde{G}$, then  $(\bigcup\widetilde{A},\bigcup\widetilde{B})$ is a separation of~$G$. 
    Moreover, if there is an edge between~${\bigcup \widetilde{A} \setminus \bigcup \widetilde{B}}$ and~${\bigcup \widetilde{B} \setminus \bigcup \widetilde{A}}$, then there is an edge between~${A \setminus B}$ and~${B \setminus A}$.
    So in particular, if~$(\widetilde{A},\widetilde{B})$ is a separation of~$\widetilde{G}$, then~$(\bigcup\widetilde{A},\bigcup\widetilde{B})$ is a separation of~$G$, and if~$(A,B)$ is a separation of~$G$, then~$(\widetilde{A},\widetilde{B})$ is a separation of~$\widetilde{G}$. 

    So let~${(A,B)}$ be a separation of~$G$ with~${\theta(G[A \cap B]) = k}$, and let~${\{ K^1, \dots, K^k \}}$ be a clique-cover of~${G[A \cap B]}$. 
    Then by \cref{lem:clique-quotient-graph}, ${\{ \widetilde{K}^1, \dots, \widetilde{K}^k \}}$ is a clique-cover of~${\widetilde{G}[\widetilde{A} \cap \widetilde{B}}]$. 
    With \cref{cor:Gtildeproperties}\ref{item:cor:Gtildeproperties1}, if follows that~${\abs{\widetilde{A} \cap \widetilde{B}} \leq k \cdot \reduced{\omega}{G}}$. 

    Finally, let~${(\widetilde{A},\widetilde{B})}$ be a separation of~$\widetilde{G}$ of order~$k$. 
    Since~$G[\widetilde{x}]$ is a clique for each~$\widetilde{x} \in \widetilde{A} \cap \widetilde{B}$, ${\{G[\widetilde{x}] \mid x \in \widetilde{A} \cap \widetilde{B} \}}$ is a clique-cover of~${A \cap B}$ of size~$k$, so~${\theta(G[A \cap B]) \leq k}$.
\end{proof}

\begin{lemma}
    \label{lemma:quotientpahts}
    Let~$\ell$ be a positive integer.
    For every graph~$G$, 
    \begin{enumerate}
        [label=(\arabic*)]
        \item \label{item:lemma:quotientpahts1} if ${x_1 x_2 \dots x_\ell}$ is an induced path in~$G$ and~${\ell \geq 3}$, then ${\widetilde{x_1} \widetilde{x_2} \dots \widetilde{x_\ell}}$ is an induced path in~$\widetilde{G}$, and
        \item \label{item:lemma:quotientpahts2} if $\widetilde{x_1} \widetilde{x_2} \dots \widetilde{x_\ell}$, then, for each set~${\{ r_1, r_2, \dots, r_\ell \}}$ with~${r_i \in \widetilde{x_i}}$ for all~${i \in [\ell]}$, we have that~${r_1 r_2 \dots r_\ell}$ is an induced path in~$G$.
    \end{enumerate}
\end{lemma}

\begin{proof}
     \ref{item:lemma:quotientpahts2} immediately follows from (\cref{eq:ast}). 
     For \ref{item:lemma:quotientpahts1}, by (\cref{eq:ast}), it suffices to show is that no two vertices of the path are equivalent with respect to~$\sim_{\mathcal{K}(G)}$.
     This is clear for vertices that are not adjacent, so fix some arbitrary~${i,j \in [\ell]}$ with~${i < j}$ and $\abs{i-j} = 1$. 
     Since~${\ell \geq 3}$ and the path is induced, there is a~${k \in [\ell] \setminus \{i,j\}}$ such that~${x_k}$ is adjacent to exactly one of~${x_i, x_j}$.
     Hence, $x_i$ and~$x_j$ are not true twins, and therefore~${x_i \not \sim_{\mathcal{K}(G)} x_j}$. 
\end{proof}

Now, we can prove \cref{thm:hedgehogmenger}. 

\hedgehogmenger*

\begin{proof}
    Let~$G$ be a graph with ${\reduced{\omega}{G} \leq s}$ and~${\cideg{G} \leq t}$, and let~${A,B \subseteq V(G)}$. 
    By \cref{lem:cidcidegdelta}, ${\Delta(\widetilde{G}) \leq t(s-1)}$. 
    We apply \cref{thm:gartlandinducedmenger} to~$\widetilde{G}$ and~${\widetilde{A}, \widetilde{B}}$ to find, for each positive integer~$k$, either an induced $\widetilde{A}$-$\widetilde{B}$-linkage~${\widetilde{\mathcal{L}}}$ of order~$k$ or a set~$\widetilde{S}$ of size at most~${k \cdot (t(s-1) + 1)^{t^2(s-1)^2 + 1}}$ such that there does not exist an~$\widetilde{A}$-$\widetilde{B}$-path in~$\widetilde{G} - \widetilde{S}$. 

    Let~${\widetilde{\mathcal{L}}}$ be an $\widetilde{A}$-$\widetilde{B}$-linkage of order~$k$. 
    For each~$P \in \widetilde{\mathcal{L}}$, fix a set~$R_P$ of representatives of the equivalence classes in~$V(\widetilde{P})$ with the additional property that for the startvertex in~$\widetilde{A}$, we pick a representative in~$A$ and for the endvertex in~$\widetilde{B}$, we pick a representative in~$B$. 
    By \cref{lemma:quotientpahts}, each~$G[R_P]$ is an induced $A$-$B$-path. 
    Note that since~$\widetilde{\mathcal{L}}$ is an induced $\widetilde{A}$-$\widetilde{B}$-linkage, for distinct~$P,Q \in \widetilde{\mathcal{L}}$, we observe that no vertex in~$R_P$ is equivalent to a vertex in~$R_Q$. 
    So in particular, $R_P$ and~$R_Q$ are disjoint, and by \cref{obs:GtildeEdges}\ref{obs:GtildEedges:3}, there is no edge between~$R_P$ and~$R_Q$ in~$E(G)$. 
    Hence, ${\mathcal{L} \coloneqq \{ G[R_P] \mid P \in \widetilde{\mathcal{L}}\}}$ is an $A$-$B$-linkage in~$G$ of order~$k$. 

    Lastly, let~$\widetilde{S}$ be a set of size at most~${k \cdot (t(s-1) + 1)^{t^2(s-1)^2 + 1}}$ such that there does not exist an~$\widetilde{A}$-$\widetilde{B}$-path in~$\widetilde{G} - \widetilde{S}$. 
    Let~$(\widetilde{X},\widetilde{Y})$ be a separation of~$\widetilde{G}$ with~${\widetilde{A} \subseteq \widetilde{X}}$, ${\widetilde{B} \subseteq \widetilde{Y}}$, and~$\widetilde{S} = \widetilde{X} \cap \widetilde{Y}$. 
    By \cref{lemma:quotientseparations}, the separation $(\bigcup \widetilde{X}, \bigcup \widetilde{Y})$ is a separation with~${A \subseteq \bigcup \widetilde{X}}$, ${B \subseteq \widetilde{Y}}$, and for~${S \coloneqq \widetilde{X} \cap \widetilde{Y}}$, we have~$\theta(G[S]) \leq (t(s-1) + 1)^{t^2(s-1)^2 + 1}$, as desired. 
\end{proof}

\section{Comparison with other graph parameters}
\label{sec:parametercomp}

In this section, we further analyse how the graph parameters of clique-incidence degree and clique-incidence diversity relate to other parameters to obtain further characterisations for clique-sparse graph classes. 

\begin{lemma}
    \label{lem:parameterinequalities-1}
    For every graph~$G$ holds 
    \[
        \local{\alpha}{G} 
        \leq \local{\theta}{G}
        \leq \cideg{G}.
    \]
\end{lemma}

\begin{proof}
    To show ${\local{\alpha}{G} \leq \local{\theta}{G}}$, notice that each leaf of a maximal induced star needs to be covered by a distinct clique in any clique-cover. 
    
    For the inequality ${\local{\theta}{G} \leq \cideg{G}}$, 
    notice that every vertex~$v$ is a dominating vertex of the graph induced by the closed neighbourhood of~$v$.
    Thus, for every clique~$K$ of~$G[N_G[v]]$, we observe that~${G[V(K) \cup \{v\}]}$ is a clique as well. 
    In particular, the set~${\{ G[K \cap N_G[v]] \colon K \in \mathcal{K}(G) \textnormal{ with } v \in K \}}$ is a
    clique-cover of~${G[N_G[v]]}$. 
\end{proof}

Note that all these inequalities are tight, as witnessed by trees.
On the other hand, for each of the inequalities, there exist graphs where the difference is arbitrarily large. 
For the first inequality, consider a graph obtained from the complement of a graph, which is both triangle-free and has arbitrarily large chromatic number (cf.~Erd\H{o}s~\cite{Erdos1959}), by adding a universal vertex.
For the difference between~$\local{\theta}{}$ and~$\cideg{}$, note that~$\antimatchingKK_n$ is a graph with~${\local{\theta}{\antimatchingKK_n} = 2}$ and~${\cideg{\antimatchingKK_n} = 2^{n-1}}$.

\begin{lemma}
    \label{lem:parameterinequalities-2}
    For every graph~$G$ holds
    \[
        \reduced{\Delta}{G} \leq \local{\alpha}{G}^{\reduced{\omega}{G}}.
    \]
\end{lemma}

\begin{proof}
    Let~$G$ be a graph,
    let~${s \coloneqq \local{\alpha}{G}}$, and let~${d \coloneqq \reduced{\omega}{G}}$. 
    Consider some vertex~${\widetilde{v} \in V(\widetilde{G})}$. 
    Let $T$ be the $s$-ary tree of height~$d$ and for each node~${t \in V(T)}$, let~$P_t$ denote the unique path in~$T$ from the root~$r$ of~$T$ to~$t$.
    
    We construct a labelling~$\lambda$ of some nodes of $T$ with vertices of~$\widetilde{G}$ such that every vertex in~$N_{\widetilde{G}}(v)$ appears as a label, proving that~${\deg{\widetilde{v}} \leq s^d}$.  
    We start by defining~$\lambda(r) \coloneqq \widetilde{v}$. 
    Assume for a node~$t$ that~$\lambda(t')$ is already defined for each~${t' \in V(P_t)}$. 
    We then label a set of children of~$t$ with a maximum independent set in~${\bigcap \{N_{\widetilde{G}}(\lambda(t')) \mid t' \in V(P_t) \}}$, and leave all other children unlabelled. 
    Continuing this process, we obtain a labelled subtree~$T'$ of~$T$ with the property that for each node~${t \in V(T')}$, each~$\lambda(V(P_t))$ is a clique of~$\widetilde{G}$ and for the set~$C_t$ of children of~$t$ in~$T'$, we have that~$\lambda(C_t)$ is a maximum independent set in the common neighbourhood of~$V(P_t)$.
    
    Assume for a contradiction that some~${\widetilde{w} \in N_{\widetilde{G}}(\widetilde{v})}$ is not used as a label in $T'$. 
    Let~$t$ be maximal in the tree order such that~$\widetilde{w}$ is adjacent to~$\lambda(t')$ for every node~${t' \in V(P_t)}$. 
    If~$t$ is not a leaf of~$T'$, then, since the set~$\lambda(C_t)$ is a maximum independent set in the common neighbourhood of~$V(P_t)$, there is a child~$t'$ of~$t$ in~$T'$ such that~$\widetilde{w}$ is adjacent to~$\lambda(t')$, contradicting the maximal choice of~$t$. 
    If~$t$ is a leaf of~$T'$ but not a leaf of~$T$, then, by construction, the common neighbourhood of~$V(P_t)$ is empty, a contradiction. 
    But then, $t$ is a leaf of~$T$ and ${\lambda(V(P_t)) \cup \{\widetilde{w}\}}$ is a clique in~$\widetilde{G}$ of size~${d+1}$, a contradiction. 
    Hence, $\reduced{\Delta}{G} \leq s^{d}$, as desired. 
\end{proof}

Another auxiliary graph and parameter that may naturally come to mind in connection with these graphs is the maximum degree of the linegraph~$\linegraph{\mathcal{K}(G)}$ of the clique-hypergraph of~$G$, which we call the \emph{clique-degree}, denoted by~$\cdeg{G}$. 
We observe the following inequalities. 

\begin{lemma}
    \label{obs:CID+CIDEGvsCDEG}
    For every graph~$G$, 
    \begin{enumerate}
        [label=(\arabic*)]
        \item \label{item:obs:CID+CIDEGvsCDEG-1} ${\cdeg{G} \leq \reduced{\omega}{G} \cdot  (\cideg{G}-1)}$, 
        \item \label{item:obs:CID+CIDEGvsCDEG-2} ${\reduced{\omega}{G} \leq 2^{\cdeg{G}}}$, and 
        \item \label{item:obs:CID+CIDEGvsCDEG-3} ${\cideg{G} \leq \cdeg{G}+1}$. 
    \end{enumerate}  
\end{lemma}

\begin{proof}
    Consider a maximal clique~${K \in \mathcal{K}(G)}$, and let~$\mathcal{N}$ denote the set of maximal cliques that intersect~$K$ but are not equal to~$K$. 
    
    For \ref{item:obs:CID+CIDEGvsCDEG-1}, note that the number of maximal cliques in~$\mathcal{N}$ that intersect a fixed equivalence class is at most~$\cideg{G}-1$. 
    Hence, ${\abs{\mathcal{N}} \leq \reduced{\omega}{G} \cdot (\cideg{G}-1)}$. 
    
    For \ref{item:obs:CID+CIDEGvsCDEG-2}, note that the intersection of every equivalence class of~$\sim_{\mathcal{K}(G)}$ with~$K$ is described precisely by a subset of~$\mathcal{N}$ that contains the class. 
    Hence, the clique-incidence-diversity of~$K$ is at most~$2^{\cdeg{G}}$. 
    
    Moreover, for \ref{item:obs:CID+CIDEGvsCDEG-3}, each vertex of~$K$ is contained in~$K$ and at most~$\cdeg{G}$ many other maximal cliques, hence~${\cideg{G} \leq \cdeg{G}+1}$.
\end{proof}

As we observed above,~$\local{\alpha}{}$, $\local{\theta}{}$, and~$\cideg{}$ are not equivalent. 
But with \cref{lem:parameterinequalities-1,lem:parameterinequalities-2}, as well as \cref{cor:cid-cideg-Deltatilde}, we conclude the following equivalences of parameters.

\begin{corollary}
    \label{cor:parameterequivalence1}
    The parameters~$\reduced{\Delta}{}$, 
    $\cdeg{}$, 
    $\maxP{\cideg{}, \reduced{\omega}{}}{}$,
    $\maxP{\local{\alpha}{}, \reduced{\omega}{}}{}$, and
    $\maxP{\local{\theta}{}, \reduced{\omega}{}}{}$
    are equivalent. 
    \qed
\end{corollary}

\section{An induced grid theorem on graphs with bounded clique-incidence-diversity and bounded local independence number}

Next, we prove the generalisation of \cref{thm:maxdegreegridtheorem} to graphs of bounded $\maxP{\reduced{\omega}{},\local{\alpha}{}}{}$. 
To that end, we first investigate how certain width parameters relate between~$G$ and~$\widetilde{G}$. 

\begin{lemma}
    \label{lemma:width-comparison}
    For every graph~$G$,
    \begin{enumerate}
        [label=(\arabic*)]
        \item \label{item:lemma:width-comparison0} ${\ptw{\alpha}{G} = \ptw{\alpha}{
        \widetilde{G}}}$ and~${\ptw{\theta}{G} = \ptw{\theta}{\widetilde{G}}}$, 
        \item \label{item:lemma:width-comparison1} $\ptw{\alpha}{G} \leq \ptw{\theta}{G} \leq \reduced{\tw{}}{G}$, 
        \item \label{item:lemma:width-comparison2} $\reduced{\tw{}}{G} \leq \ptw{\theta}{G} \cdot \reduced{\omega}{G}$, and 
        \item \label{item:lemma:width-comparison3} $\reduced{\tw{}}{G} \leq \ptw{\alpha}{G} \cdot \reduced{\Delta}{G}$. 
    \end{enumerate}
\end{lemma}

\begin{proof}
    For \ref{item:lemma:width-comparison0} we first look at a way to turn a tree-decomposition of~$G$ into one of~$\widetilde{G}$ and vice versa. 
    Moreover, both of these constructions preserve both the $\alpha$-width and the $\theta$-width of the original tree-decomposition. 
    
    Let~$(T,\beta)$ be a tree-decomposition of~$G$. 
    Then we define~$\widetilde{\beta}$ by setting~${\widetilde{\beta}(t) \coloneqq \{ \widetilde{v} \mid v \in \beta(t) \}}$. 
    We claim that~${(T,\widetilde{\beta})}$ is a tree-decomposition of~$\widetilde{G}$. 
    Clearly, the first two properties of a tree-decomposition follow directly from the fact that~$(T,\beta)$ is a tree-decomposition. 
    So consider a vertex~${\widetilde{v} \in V(\widetilde{G})}$ and assume for a contradiction that~${\{ t \in V(T) \mid v \in \widetilde{\beta}(t) \}}$ induces a disconnected subforest of~$T$. 
    Then there is a node~$t'$ of~$T$ such that~$\widetilde{v}$ is not in~$\widetilde{\beta}(t)$, but two distinct components of~${T - t'}$ contain nodes~$t_1$ and~$t_2$ with~${\widetilde{v} \in \widetilde{\beta}(t_1) \cap \widetilde{\beta}(t_2)}$. 
    So $\beta(t_1)$ contains a vertex~$v_1$ and~$\beta(t_2)$ contains a vertex~$v_2$, both equivalent to~$v$, while~$\beta(t')$ does not contain any node equivalent to~$v$. 
    In particular, the subtree induced by~${\{ t \in V(T) \mid v_1 \in \beta(t) \}}$ and the subtree induced by~${\{ t \in V(T) \mid v_2 \in \beta(t) \}}$ are contained in distinct components of~${T - t'}$. 
    But then, no bag of~$(T,\beta)$ contains the edge~$v_1v_2$, contradicting that~$(T,\beta)$ is a tree-decomposition. 
    It is easy to see that~$\widetilde{G}[\widetilde{\beta}(t)]$ is isomorphic to~$\widetilde{G[\beta(t)]}$. 
    So the fact that the $\alpha$- and $\theta$-width of~$(T,\beta)$ and~$(T,\widetilde{\beta})$ are the same follows from \cref{cor:Gtildeproperties}. 

    On the other hand, let~$(T,\beta)$ be a tree-decomposition of~$\widetilde{G}$. 
    We define~$\overline{\beta}$ by setting~${\overline{\beta}(t) \coloneqq \{ v \mid \widetilde{v} \in \beta(t) \}}$. 
    Now~${(T,\overline{\beta})}$ is a tree-decomposition of~$G$ as all the properties follow directly from the fact that~$(T,\beta)$ is a tree-decomposition.
    Similar to before, $\widetilde{G}[\beta(t)]$ is isomorphic to~$\widetilde{G[\overline{\beta}(t)]}$.
    So the fact that the $\alpha$- and $\theta$-width of~$(T,\beta)$ and~$(T,\overline{\beta})$ are the same follows again from \cref{cor:Gtildeproperties}. 

    Both constructions together show \ref{item:lemma:width-comparison0}. 

    Towards \ref{item:lemma:width-comparison1}, note that since the size of a maximum independent set is at most as large as a smallest clique-cover, we observe that~$\ptw{\alpha}{G} \leq \ptw{\theta}{G}$. 
    And since clearly~$\ptw{\theta}{\widetilde{G}} \leq \tw{\widetilde{G}}$, \ref{item:lemma:width-comparison1} follows from~\ref{item:lemma:width-comparison0}. 

    Now \ref{item:lemma:width-comparison2} and \ref{item:lemma:width-comparison3} follow with~\ref{item:lemma:width-comparison0} directly from the facts that~${\abs{V(G)} \leq \alpha(G) \Delta(G)}$ and~${\abs{V(G)} \leq \theta(G) \omega(G)}$ for every graph~$G$. 
\end{proof}

Using the established relations, we finally prove \cref{thm:hedgehoggrid}. 

\hedgehoggrid*

\begin{proof}
    Let $f'$ be the function from \cref{thm:maxdegreegridtheorem}. 
    We define~${f(k,s,d) \coloneqq f'(k,(s-1) \cdot d)}$. 
    Let~$G$ be a graph with~${\ptw{\theta}{G} \geq f(k,s,d)}$.
    By \cref{lemma:width-comparison}, ${\ptw{\theta}{G} = \ptw{\theta}{\widetilde{G}} \leq \reduced{\tw{}}{G} = \tw{\widetilde{G}}}$.
    By \cref{lem:parameterinequalities-2}, ${\reduced{\Delta}{G} = \Delta(\widetilde{G}) \leq \local{\alpha}{G}^{\reduced{\omega}{G}}}$. 
    So by applying \cref{thm:maxdegreegridtheorem} to~$\widetilde{G}$, we find a $(k \times k)$-grid as an induced minor of~$\widetilde{G}$. 
    And since, by \cref{lem:clique-quotient-graph}, $\widetilde{G}$ is isomorphic to an induced subgraph of~$G$, we find a $(k \times k)$-grid as an induced minor of~$G$ as well. 
\end{proof}

\section{Forbidden induced subgraphs}
\label{sec:forbiddeninducedsubgraphs}

We move on to proving the arguably most powerful theorem of this paper:
An asymptotic classification of the hereditary graph classes with bounded $\maxP{\cideg{},\reduced{\omega}{}}{}$ in terms of a finite set of parametric graphs that have to be excluded as induced subgraphs.
For this purpose, we first provide an asymptotic description of all hereditary graph classes where $\reduced{\omega}{}$ is bounded before obtaining the full asymptotic classification theorem as a corollary. 

First, we make some basic observations on clique-incidence-diversity. 

Let~$G$ be a graph and~${X \subseteq V(G)}$.
A set~${Y \subseteq N_G(X)}$ of neighbours of~$X$ is said to be \emph{$X$-diverse} if no two distinct vertices in~$Y$ have the same set of neighbours in~$X$, that is~${N_G(y_1) \cap X \neq N_G(y_2) \cap X}$ for all distinct choices of~${y_1,y_2 \in Y}$.
The \emph{diversity number} of a set~$X$ of vertices from a graph~$G$, denoted by~$\dn{G}{X}$, is the maximum size of an $X$-diverse set in~${N_G(X)}$.
In case the graph~$G$ is understood from the context, we drop the subscript. 

\begin{lemma}\label{lemma:diverseneighbours}
    Let~$G$ be a graph and~${K \in \mathcal{K}(G)}$, then~${\log(\dn{}{K}+2) \leq \abs{\widetilde{K}} \leq 2^{\dn{}{K}}}$.
\end{lemma}

\begin{proof}
    Let~$X$ be a maximum $K$-diverse set in~$G$. 
    Observe that for every vertex~${y \in N(K) \setminus X}$ there exists a vertex~${x \in X}$ such that~${N(x) \cap K = N(y) \cap K}$ as~$X$ is maximum. 
    It follows that the classes of $\sim_{\mathcal{K}(G)}$ within~$K$ are exactly determined by the subsets of~$X$.
    To be more precise, for two vertices~${u,v \in K}$ we have~${u \sim_{\mathcal{K}(G)} v}$ if and only if~${N(u) \cap X = N(v) \cap X}$. 
    Therefore, $\abs{\widetilde{K}} \leq 2^{\abs{X}}=2^{\dn{}{K}}$.
    
    For the lower bound, simply observe that for every~${x \in X}$ the set~${(N(x) \cap K) \cup \{ x \}}$ belongs to some maximal clique of~$G$ which cannot contain any other vertex of~$X$. 
    Hence, we obtain~${\abs{X} = \dn{}{K}}$ many pairwise distinct maximal cliques of~$G$, all of which have a non-empty intersection with~$K$, and all these intersections are pairwise distinct. 
    Since~$K$ is maximal, $K$ is not a subset of any of these cliques. 
    Hence, $\widetilde{K}$ has at least~$\dn{}{K}$ proper non-empty subsets, so ${\dn{}{K} + 2 \leq 2^{\abs{\widetilde{K}}}}$, as required.
\end{proof}

To derive our forbidden induced subgraphs with the help of \cref{lemma:diverseneighbours}, we employ a couple of seminal results from Ramsey theory and the theory of vertex-minors. 

We consider all matrices over~$\mathds{F}_2$, and denote the \emph{adjacency matrix} of a graph~$G$ by~$\AM{G}$. 
For a set of vertices~$X$ in a graph~$G$ we define its \emph{cutrank}, denoted~$\cutrank{G}{X}$, as the rank of the submatrix of~$\AM{G}$ with rows~$X$ and columns~${\V{G}\setminus X}$. 
Given two disjoint sets of vertices~${X, Y \subseteq V(G)}$ we denote the \emph{local cutrank}, that is~${\cutrank{G[X \cup Y]}{X} = \cutrank{G[X \cup Y]}{Y}}$, by~${\localcutrank{G}{X,Y}}$. 
If the graph~$G$ is clear from the context, we drop the subscript.

Let~${X \subseteq V(G)}$ be some set of vertices in a graph~$G$.
Observe that the submatrix of~$\AM{G}$ with rows~$X$ and columns~${V(G) \setminus X}$ has exactly~$\dn{}{X}$ many distinct rows. 
Thus, the following lemma allows us to relate the diversity number of maximal cliques to their local cutrank. 

\begin{lemma}[\cite{NguyenO2020-averangecutrank}]
    \label{lemma:averagecut}
    Let~$M$ be a matrix over~$\mathds{F}_2$.
    Then the number of distinct rows in~$M$ is at most~${2^{\mathsf{rank}(M)}}$.
\end{lemma}

Indeed, combining \cref{lemma:averagecut} with \cref{lemma:diverseneighbours} yields the following.

\begin{corollary}
    \label{cor:diversitytorank}
    Let~$G$ be a graph and let~${K \in \mathcal{K}}$ be a maximal clique in~$G$, then ${\log(\log(\abs{\widetilde{K}})) \leq \localcutrank{}{K,N(K)}}$.
\end{corollary}

Let~$B$ be a bipartite graph with a bipartition~${V_1 \cup V_2}$ such that both~$V_i$ induce independent sets. 
The \emph{biadjacency matrix} of~$B$ with respect to~$V_1$ and~$V_2$ is the submatrix of~$\AM{G}$ with rows~$V_1$ and columns~$V_2$.
The definitions of matchings ($\mathscr{M}$), anti-matchings ($\mathscr{A}$), and half-graphs ($\mathscr{H}$) provide natural bipartitions of their vertex sets, and their biadjacency matrices are unique up to reordering of rows and columns.

Let $G$ be a graph and ${X, Y \subseteq V(G)}$ be two disjoint sets of vertices.
We say that $X$ and $Y$ are \emph{coupled} if the submatrix of $\AM{G}$ with rows $X$ and columns $Y$ is the biadjacency matrix of a matching, anti-matching, or half-graph. 
Ding, Oporowski, Oxley, and Vertigan \cite{DingOOV1996-Unavoidable3-con} showed that any pair of sets with large cutrank between them must contain a large coupled pair of sets.

\begin{theorem}[\cite{DingOOV1996-Unavoidable3-con}]
    \label{thm:coupledpairs}
    There exists a function ${f_{\ref{thm:coupledpairs}} \colon \N \to \N}$ such that for every positive integer~$k$, every graph~$G$ and every pair of disjoint sets~${X,Y \subseteq V(G)}$ with~${\localcutrank{}{X,Y} \geq f_{\ref{thm:coupledpairs}}(k)}$ there exists a coupled pair of sets~${X' \subseteq X}$ and~${Y' \subseteq Y}$ with~${\abs{X'} = \abs{Y'} = k}$.
\end{theorem}

The last bit of technology we need are the standard \emph{Ramsey numbers} \cite{Ramsey1929}. 
We denote by~$R(k)$ the Ramsey number that guarantees that every graph of size at least~$R(k)$ contains a clique or an independent set of size at least~$k$.

With this, we are ready to state and prove the main result of this \namecref{sec:forbiddeninducedsubgraphs}.

\begin{theorem}\label{thm:obstructionsforcid}
    Let $\mathfrak{A}$ denote the set~${\{\matchingKI,\matchingKK,\antimatchingKI,\antimatchingKK,\halfgraphKI,\halfgraphKK\}}$.
    Then the graph parameters $\bigPG{\mathfrak{A}}{\InducedSubgraph}{}$ and $\reduced{\omega}{}$ are equivalent.
\end{theorem}

\begin{proof}
    Let~${\mathscr{X} \in \mathfrak{A}}$ and~$t$ be a positive integer.
    Notice that~$\mathscr{X}_t$ has a natural bipartition of its vertex set, induced by the definition, into a clique~$X$ and a set~$Y$ which is either a clique or an independent set such that~$X$ and~$Y$ are coupled and both of size~$t$. 
    Moreover, note that in all cases, ${\widetilde{\mathscr{X}_t} = \mathscr{X}_t}$ since~$\mathscr{X}_t$ has no true twins. 
    It is straight forward to observe that $\reduced{\omega}{}$ is monotone under the induced subgraph relation, and so if~$\mathscr{X}_t$ is an induced subgraph of a graph~$G$, then~${\reduced{\omega}{G} \geq \reduced{\omega}{\mathscr{X}_t} = \omega(\mathscr{X}_t) \geq t}$.
    Hence, we obtain~$\bigPG{\mathfrak{A}}{\InducedSubgraph}{G} \leq \reduced{\omega}{G}$ for all graphs~$G$.
    
    To prove the other direction, let~$G$ be a graph satisfying
    $\reduced{\omega}{G} \geq 2^{2^{f_{\ref{thm:coupledpairs}}(R(t))}}$. 
    We claim that there exists some~${\mathscr{X} \in \mathfrak{A}}$ such that~$G$ contains~$\mathscr{X}_t$ as an induced subgraph. 
    
    Let~${K \in \mathcal{K}(G)}$ be some maximal clique of~$G$ with~${\abs{\widetilde{K}}} \geq 2^{2^{f_{\ref{thm:coupledpairs}}(R(t))}}$ and let~${X \coloneq N(K)}$ be the neighbourhood of~$K$. 
    By \cref{cor:diversitytorank} this implies~${\localcutrank{}{K,X} \geq f_{\ref{thm:coupledpairs}}(R(t))}$. 
    We now call upon \cref{thm:coupledpairs} to find sets~${Y \subseteq K}$ and~${Z \subseteq X}$ such that~$Y$ and~$Z$ are coupled and~${\abs{Y} = \abs{Z} = R(t)}$.
    Let~$H$ be the subgraph of~$G$ induced by the edges\footnote{That is, the vertex set of~$H$ is~${Y \cup Z}$ and the edge set of~$H$ is exactly those edges with exactly one end in~$Y$ and the other end in~$Z$.} between~$Y$ and~$Z$.
    Notice that there exists a bijective function~${\phi \colon Z \to Y}$ such that for each~${z \in Z}$, ${\phi(z)z \in E(G)}$ if~$H$ is a matching, ${\phi(z)z \notin E(G)}$ if~$H$ is an anti-matching, and, if~$H$ is a half-graph, $\phi$ ensures that, for the ordering of the sets~${Y = \{ y_1,y_2,\dots,y_{\abs{Y}}\}}$ and~${Z = \{ z_1,z_2,\dots,z_{\abs{Z}}\}}$ induced by the definition of half-graphs, it holds that~${\phi(z_i) = y_i}$. 
    As~${\abs{Z} = R(t)}$ we find a set~${Z' \subseteq Z}$ which is either an independent set or a clique. 
    Let~${Y' \coloneqq \{ \phi(z) \mid z \in Z' \}}$. 
    Notice that~$Y'$ and~$Z'$ are again coupled.
    Hence, ${G[Y' \cup Z']}$ is isomorphic to one of~$\matchingKI_t$, $\matchingKK_t$, $\antimatchingKI_t$, $\antimatchingKK_t$, $\halfgraphKI_t$ or~$\halfgraphKK_t$, as desired.
\end{proof}

As a direct consequence, we obtain a forbidden subgraph characterisation for bounding~$\maxP{\cideg{},\reduced{\omega}{}}{}$ (or any of the equivalent parameters from \cref{cor:parameterequivalence1}). 

\begin{corollary}
    \label{cor:parameterequivalence2}
    Let $\mathfrak{B}$ denote the set~${\{ \matchingKI,\matchingKK,\antimatchingKK,\halfgraphKK \} \cup \{ \Star \}}$.
    Then the graph parameter $\bigPG{\mathfrak{B}}{\InducedSubgraph}{}$ is equivalent to each of~$\reduced{\Delta}{}$, 
    $\cdeg{}$, 
    $\maxP{\cideg{}, \reduced{\omega}{}}{}$,
    $\maxP{\local{\alpha}{}, \reduced{\omega}{}}{}$, and
    $\maxP{\local{\theta}{}, \reduced{\omega}{}}{}$. 
\end{corollary}

\begin{proof}
    Let~$\mathfrak{A}$ be as in~\cref{thm:obstructionsforcid}, and let~${\mathfrak{C} \coloneqq \mathfrak{A} \cup \{ \Star \}}$. 

    First, let~$t$ be a positive integer and let~$G$ be a graph such that~${\maxP{\local{\alpha}{},\reduced{\omega}{}}{G} \leq t}$. 
    We obtain from \cref{thm:obstructionsforcid} that~${\bigPG{\mathfrak{A}}{\InducedSubgraph}{G} \leq f(t)}$ for some function~$f$, for which we may assume without loss of generality, that~${f(t) \geq t}$. 
    Since~${\bigPG{\Star}{\InducedSubgraph}{G} = \local{\alpha}{G} \leq t}$, we additionally observe that~${\bigPG{\mathfrak{C}}{\InducedSubgraph}{G} \leq f(t)}$, and since~${\mathfrak{B} \subseteq \mathfrak{C}}$, we conclude~${\bigPG{\mathfrak{B}}{\InducedSubgraph}{G} \leq f(t)}$, as desired. 

    On the other hand, let~$t$ be a positive integer and let~$G$ be a graph such that~${\bigPG{\mathfrak{B}}{\InducedSubgraph}{G} \leq t}$. 
    Note that since both~$\antimatchingKI_{t+1}$ and~$\halfgraphKI_{t+1}$ contain~$\Star_{t+1} = K_{1,t+1}$ as an induced subgraph, it immediately follows that~${\bigPG{\mathfrak{A}}{\InducedSubgraph}{G} \leq t}$. 
    Hence, as seen in the proof of \cref{thm:obstructionsforcid}, ${\reduced{\omega}{G} \leq t}$. 
    But since~${\bigPG{\Star}{\InducedSubgraph}{G} = \local{\alpha}{G} \leq t}$, we obtain that~${\maxP{\local{\alpha}{},\reduced{\omega}{}}{G} \leq t}$. 
    
    By \cref{cor:parameterequivalence1}, all the desired equivalences follow. 
\end{proof}

Note that this corollary now finishes the proof of \cref{thm:clique-sparse}. 

\section{%
\texorpdfstring{Rankwidth versus $\boldsymbol{\alpha}$-treewidth}{Rankwidth versus alpha-treewidth}}
\label{sec:rankwidth}

In this section, we consider classes obtained by excluding some parametric graphs and how this affects the relation between rankwidth and $\alpha$-treewidth.
We start with defining the rankwidth of a graph, a concept first introduced by Oum and Seymour~\cite{OumS2006-cliquewidth-branchwidth-rankwidth}. 

A \emph{rank-decomposition} of a graph~$G$ is a tuple $\Brace{T,\delta}$, where~$T$ is a cubic tree and~$\delta$ is a bijection between the leaves of~$T$ and the vertices of~$G$. 
For every edge~$e$ of~$T$, note that the two subtrees of~${T-e}$ induce a bipartition of the leaves of~$T$. 
Let~$X_e$ denote the image of one of those sets of leaves under~$\delta$. 
The \emph{cutrank} of~$e$ is defined as the cutrank of~$X_e$, i.e.~$\cutrank{G}{X_e}$ (see \cref{sec:forbiddeninducedsubgraphs} for a definition of the cutrank of a set of vertices). 
Note that this is well-defined as~$\cutrank{G}{X_e} = \cutrank{G}{V(G) \setminus X_e}$. 
The \emph{width} of $\Brace{T,\delta}$ is defined as the maximum cutrank of one of the edges of~$T$ and the \emph{rankwidth} of~$G$, denoted~$\rankwidth{G}$, is the minimum width over all rank decompositions of~$G$.

We observe the following. 
\begin{observation}
    \label{obs:cutrank}
    Let~$G$ be a graph. For every~${X \subseteq V(G)}$, we have~$\cutrank{G}{\bigcup \widetilde{X}} = \cutrank{\widetilde{G}}{\widetilde{X}}$. 
\end{observation}

As a first result of this section, we observe that bounding both the rankwidth and the local independence number is sufficient to bound the $\alpha$-treewidth of a graph.

\begin{theorem}
    \label{thm:rw-vs-alpha-tw1}
    For every positive integers~$d$ and~$k$ and every graph~$G$ with~${\local{\alpha}{G} \leq d}$ and~${\rankwidth{G} \leq k}$ we have~${\ptw{\alpha}{G} \leq 3dk}$. 
\end{theorem}

\begin{proof}
    Let~$G$ be a graph with~${\local{\alpha}{G} \leq d}$ and~${\rankwidth{G} \leq k}$. 
    Let~$(T,\delta)$ be a rank-decomposition of~$G$ of width at most~$k$. 
    We define a tree-decomposition~$(T,\beta)$ of~$G$ as follows. 
    For each leaf~$t$ of~$T$, we set~${\beta(t) \coloneqq \delta(t)}$. 
    For each internal node~$t$ of~$T$, consider the sets~$X_t^1$, $X_t^2$, $X_t^3$ of vertices mapped to the leaves of the three components of~${T - t}$, and define~${\beta(t)}$ as the set of those vertices~$v$ in~$X_t^1 \cup X_t^2 \cup X_t^3$ for which there are~$i,j \in [3]$ with~${i \neq j}$ and a neighbour~${w \in N_G(v)}$ such that~${v \in X_t^i}$ and~${w \in X_t^j}$. 
    We claim that~$(T,\beta)$ is a tree-decomposition of~$G$. 
    Clearly, ${\bigcup_{t \in V(T)} \beta(t) = V(G)}$. 
    Moreover, for every edge~${vw \in E(G)}$, both~$v$ and~$w$ are contained in~$\beta(t)$ for any internal node~$t$ for which~$\delta^{-1}(v)$ and~$\delta^{-1}(w)$ are in different components of~${T - t}$. 
    Since the latter holds in particular for the neighbour~$t'$ of~$\delta^{-1}(v)$ in~$T$, we have~${vw \subseteq \beta(t')}$. 
    Lastly, suppose that~${v \in \beta(t_1) \cap \beta(t_2)}$ for two non-adjacent nodes~$t_1$ and~$t_2$ of~$T$, and let~$t_3$ be any node of~$T$ on the unique path between~$t_1$ and~$t_2$ in~$T$. 
    Consider~${(T - t_1) - t_2}$, which has five components, two adjacent to only~$t_1$, two adjacent to only~$t_2$, and one adjacent to both~$t_1$ and~$t_2$. 
    For~${i \in [2]}$, let~$Z_i$ denote the union of the components adjacent to only~$t_i$, and let~$Z_3$ denote the component adjacent to both~$t_1$ and~$t_2$, which contains~$t_3$. 
    If~$\delta^{-1}(v)$ is contained in~$Z_i$ for~${i \in [2]}$, then, since~${v \in \beta(t_{2-i})}$, there is a neighbour~$w$ of~$v$ with~${\delta^{-1}(w) \in V(Z_{2-i})}$. 
    And since both~$Z_1$ and~$Z_2$ are contained in distinct components of~${T - t_3}$, we conclude~${v \in \beta(t_3)}$. 
    So assume~$\delta^{-1}(v)$ is contained in~$Z_3$. 
    Then, since~${v \in \beta(t_1) \cap \beta(t_2)}$, there are neighbours~$w_1$ and~$w_2$ of~$v$ with~${\delta^{-1}(w_1) \in V(Z_1)}$ and~${\delta^{-1}(w_2) \in V(Z_2)}$. 
    And since both~$Z_1$ and~$Z_2$ are contained in distinct components of~${T - t_3}$, one of them is distinct to the component of~${T - t_3}$ containing~$\delta^{-1}(v)$. 
    Therefore,~${v \in \beta(t_3)}$.
    Thus we conclude that~${(T,\beta)}$ is a tree-decomposition of~$G$. 

    Finally, we prove that~${\alpha(G[\beta(t)]) \leq 3dk}$ for each node~${t \in V(T)}$. 
    Since the bags of leaves of~$T$ have size~$1$, they have independence number~$1$ as well. 
    For an internal node~$t$, again consider~$X_t^1$, $X_t^2$, $X_t^3$ as above, and let~$A$ be a maximum independent set of~$G[\beta(t)]$. 
    Suppose for a contradiction that~${\abs{A} \geq 3dk+1}$. 
    By the pigeonhole principle, there is an~${i \in [3]}$ such that~${\abs{A \cap X_t^i} \geq dk+1}$, say~${i=1}$. 
    By definition, every~${v \in X_t^1}$ has a neighbour in~${X_t^2 \cup X_t^3}$. 
    Since~${\local{\alpha}{G} \leq d}$, any vertex in~${X_t^2 \cup X_t^3}$ can have at most~$d$ vertices of~$A$ in its neighbourhood. 
    Let~${B \subseteq X_t^2 \cup X_t^3}$ be minimal such that~${A \subseteq N_G(B)}$. 
    By the observation above, ${\abs{B} \geq k+1}$. 
    But by the minimality of~$B$, each~${w \in B}$ has a neighbour in~$A$ that is not adjacent to any other~${w'\in B}$. 
    Hence, there is a matching of size~$k+1$ between~$X_t^1$ and~${X_t^2 \cup X_t^3}$. 
    But such a matching witnesses that~$\cutrank{G}{X_t^1} \geq k+1$, contradicting that the rankwidth of~$G$ is at most~$k$.
\end{proof}

The complete bipartite graph~$K_{n,n}$ is a graph of rankwidth~$1$ and $\alpha$-treewidth~$n$. 
The $(n \times n)$-grid is a graph of local-independence number~$4$ and $\alpha$-treewidth roughly~$n/2$.
To see this let $n\geq 3$ and $G_n$ be the $(n \times n)$-grid.
Notice that $G_n=\widetilde{G_n}$ since $G_n$ is a bipartite graph without true twins.
It now follows from \cref{lemma:width-comparison} that $\ptw{\alpha}{G_n}\leq \reduced{\tw{}}{G} = \tw{G} \leq 2\cdot \ptw{\theta}{G_n}$.
Moreover, since bipartite graphs are perfect, it also holds that $\theta(H)=\alpha(H)$ for all induced subgraphs $H$ of $G_n$.
Hence $\ptw{\alpha}{G_n}=\ptw{\theta}{G_n}$.
With $\tw{G_n}=n$ it follows that $n/2 \leq \ptw{\alpha}{G_n}$.
Moreover, an optimal tree-decomposition of $G_n$ can easily be observed to be of independence-number at most $\lceil n/2\rceil$, which completes the argument.

So, in general, it is necessary to bound both of these parameters, namely the rankwidth and the local independence number, in order to bound the $\alpha$-treewidth as well. 

On the other hand, for every~$n$ there is a graph of rankwidth~$n$, $\alpha$-treewidth~$2$, and local independence number~$3$, as we see below. 

As the main result of this section, we prove the following theorem.

\begin{theorem}
    \label{thm:rw-vs-alpha-tw2}
    Let~$\mathfrak{B}$ denote the set~${\{ \matchingKI, \matchingKK, \antimatchingKK, \halfgraphKK, \Star \}}$, 
    and let~$\mathcal{F} \subseteq \bigcup \mathfrak{B}$ be a set of graphs, each of which is contained in at least one~${\mathcal{X} \in \mathfrak{B}}$. 
    The parameters~$\maxP{\local{\alpha}{},\rankwidth{}}{}$ and~$\ptw{\alpha}{}$ are equivalent on the class of $\mathcal{F}$-free graphs if and only if
    \begin{itemize}
        \item $\mathcal{F}$ contains~${\matchingKI_2 \cong P_4}$ and~$\Star_n$ for some~${n \in \N}$, or
        \item $\mathcal{F}$ contains at least one graph from each~${\mathcal{X} \in \mathfrak{B}}$. 
    \end{itemize}
\end{theorem}

We prove this theorem in two steps. 
First, we analyse the relationship between rankwidth and $\alpha$-treewidth on graphs of bounded~$\maxP{\local{\alpha}{},\reduced{\omega}{}}{}$. 

\begin{theorem}
    \label{thm:rw-vs-alpha-tw3}
    Let~$d$ be a positive integer and let~$\mathcal{G}$ be a hereditary graph class such that~${\maxP{\local{\alpha}{},\reduced{\omega}{}}{G} \leq d}$ for every~${G \in \mathcal{G}}$. 
    Then the parameters~$\ptw{\alpha{}}{}$, $\rankwidth{}$ and~$\reduced{\tw{}}{}$ are equivalent on~$\mathcal{G}$. 
\end{theorem}

\begin{proof}
    Clearly, $\ptw{\alpha}{G} \leq \tw{G}$ for every graph~$G$. 
    Oum~\cite{Oum2008-rw-bw} showed that~${\rankwidth{G} \leq \tw{G}+1}$ for every graph~$G$.  
    From \cref{lemma:width-comparison}, we know that~$\ptw{\alpha}{G} = \ptw{\alpha}{\widetilde{G}}$ for every graph~$G$. 
    Similarly, we claim that~${\rankwidth{G} = \rankwidth{\widetilde{G}}}$, unless $\widetilde{G}$ is edgeless but~$G$ is not (in which case ${\rankwidth{G} = 1}$ but ${\rankwidth{\widetilde{G}} = 0}$).
    Indeed, clearly~${\rankwidth{\widetilde{G}} \leq \rankwidth{G}}$. 
    For the other direction, consider a rank-decomposition~$(T,\delta)$ of~$\widetilde{G}$ of width~$\rankwidth{\widetilde{G}}$. 
    For a set~${X \subseteq V(G)}$, let~$T_{{X}}$ denote a subcubic tree with exactly~$\abs{X}$ leaves and exactly one node of degree~$2$. 
    For every leaf~$t$ of~$T$ with~$\delta(t) = \widetilde{v}$, we extend the rank-decomposition by identifying the degree-$2$ vertex of~$T_{\widetilde{v}}$ with~$t$ and label the new leaves with an arbitrary bijection to~$\widetilde{v}$. 
    Let~$(T',\delta')$ denote the resulting rank-decomposition of~$G$. 
    For every edge of~$T'$ that corresponds to an edge of~$T_{\widetilde{v}}$ for some~${\widetilde{v} \in V(\widetilde{G})}$, we observe that one component~${T-e}$ is mapped by~$\delta'$ to a set of true twins, and hence the cutrank of~$e$ is at most~$1$. 
    For every edge~$e'$ of~$T'$ that corresponds to an edge~$e$ of~$T$, if~$\widetilde{X}$ and~$\widetilde{Y}$ denotes the images of the bipartition of the leaves of~$T$ induced by~$e$ under $\delta$, then the image of the bipartition of the leaves of~$T'$ induced by~$e'$ under $\delta'$ are~$\bigcup \widetilde{X}$ and~$\bigcup \widetilde{Y}$. 
    Hence, the cutrank of~$e$ in~$(T,\delta)$ is equal to the cutrank of~$e'$ in~$(T',\delta')$ by \cref{obs:cutrank}, and the rankwidth of~$G$ is at most~${\max\{1,\rankwidth{\widetilde{G}}\}}$. 
    So in total, we conclude that $\ptw{\alpha}{G} \leq \reduced{\tw{}}{G}$ and~${\rankwidth{G} \leq \reduced{\tw{}}{G} + 1}$. 

    Now consider a graph~${G \in \mathcal{G}}$ and set~${w \coloneqq \ptw{\alpha}{G}}$. 
    Let~$(T,\beta)$ be a tree-decomposition of~$\widetilde{G}$ of $\alpha$-width~$w$. 
    Since~${\reduced{\omega}{\widetilde{G}[\beta(t)]} \leq \reduced{\omega}{G} \leq d}$ for every node~${t \in V(T)}$, by Ramsey's theorem we have that~$\abs{\beta(t)} \leq R(\max\{d,w\})$. 
    Hence, $\reduced{\tw{}}{G} \leq R(\max\{d,w\})$, as desired. 

    Now consider a graph~${G \in \mathcal{G}}$ and set~${k \coloneqq \rankwidth{G}}$. 
    By \cref{thm:rw-vs-alpha-tw1}, we have $\ptw{\alpha}{G} \leq 3dk$, and hence, by the deduction above, ${\reduced{\tw{}}{G} \leq R(3dk)}$. 
\end{proof}

For the second part of \cref{thm:rw-vs-alpha-tw2}, we define the following graphs. 
Let~${n \geq 1}$ and~${k \geq 2}$ be integers, let~$X^1 = \{ x_1^1, \dots x_n^1 \}, \dots, X^k = \{ x_1^k, \dots x_n^k \}$ be pairwise disjoint ordered sets, each of size~$n$, and let~$\{ y^1, \dots, y^k \}$ and~$\{ z^1, \dots, z^{k-1} \}$ be sets of size~$k$ and~$k-1$, respectively, that are disjoint to each other and to each set~$X^j$. 
We define the following graphs, see also \cref{fig:iteratedHedgehogs}. 
\begin{itemize}
    \item $Q_k(\matchingII_n) = \bigcup_{i \in [k-1]} \mathscr{M}(X^i, X^{i+1}) \cup \bigcup_{i \in [k]} K(\{y^i\}, X^i)$
    \item $Q_k(\matchingKK_n) = \bigcup_{i \in [k-1]} \mathscr{M}(X^i, X^{i+1}) \cup \bigcup_{i \in [k]} K(\{y^i\} \cup X^i)$
        \item $Q_k(\matchingKI_n) = \bigcup_{i \in [k-1]} \mathscr{M}(X^i, X^{i+1}) \cup \bigcup_{i \in [ \lceil k/2 \rceil ]} K(\{y^{2i-1}\} \cup X^{2i-1})$
    \item $Q_k(\halfgraphII_n) = \bigcup_{i \in [k-1]} \mathscr{H}(X^i, X^{i+1}) \cup \bigcup_{i \in [k]} K(\{y^i\}, X^i)$
    \item $Q_k(\halfgraphKK_n) = \bigcup_{i \in [k-1]} \mathscr{H}(X^i, X^{i+1}) \cup \bigcup_{i \in [k]} K(\{y^i\} \cup X^i)$
    \item $Q_k(\antimatchingKK_n) = \bigcup_{i \in [k-1]} \mathscr{A}(X^i, X^{i+1}) \cup \bigcup_{i \in [k]} K(\{y^i\} \cup X^i) \cup \bigcup_{i \in [k-1]} K(\{z^i\}, X^i \cup X^{i+1})$
\end{itemize}

\begin{figure}[!ht]
    \centering
    \includegraphics[scale=0.5]{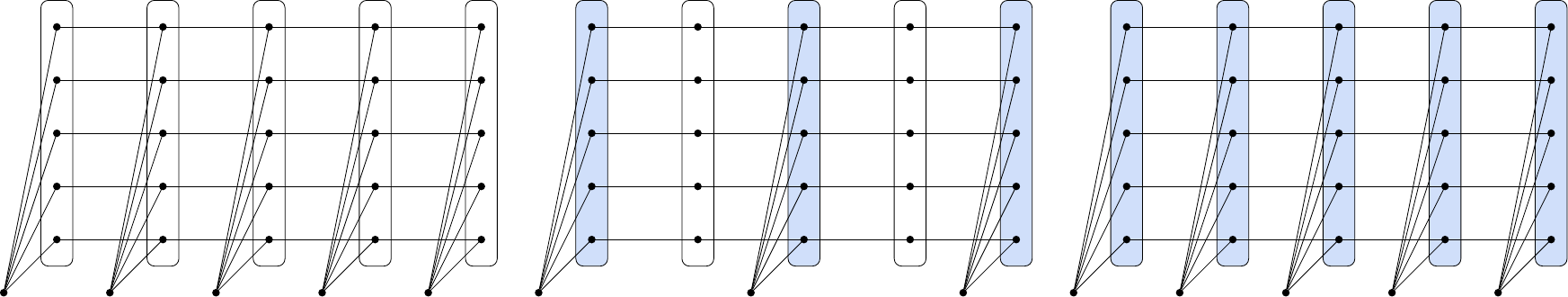}
    
    \includegraphics[scale=0.5]{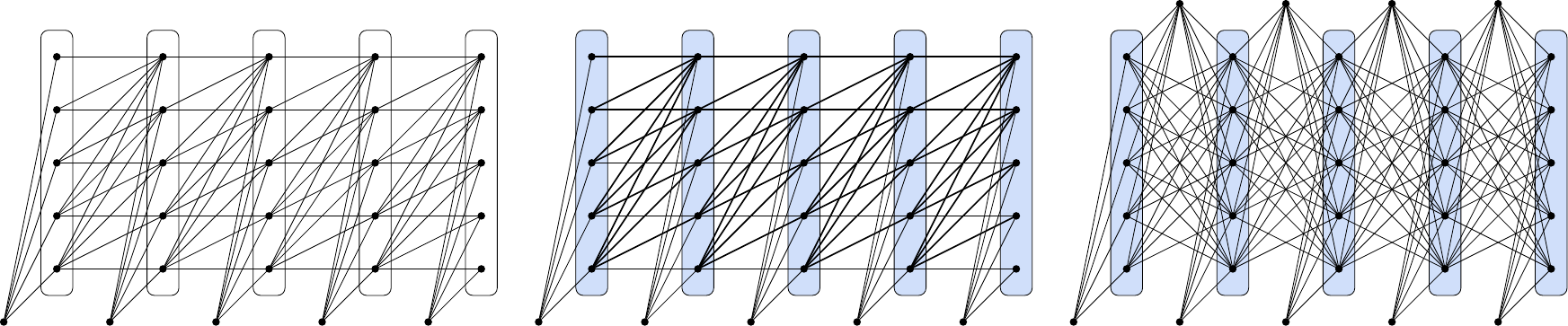}
    \caption{Top row: $Q_5(\matchingII_5)$, $Q_5(\matchingKI_5)$, and~$Q_5(\matchingKK_5)$. \\
        Bottom row: $Q_5(\halfgraphII_5)$, $Q_5(\halfgraphKK_5)$, and~$Q_5(\antimatchingKK_5)$. }
    \label{fig:iteratedHedgehogs}
\end{figure}

\begin{lemma}
    \label{lem:Q_n(X)-1}
    For each integer~${n \geq 3}$ and each~${\mathcal{X} \in \{\matchingKI, \matchingKK, \antimatchingKK, \halfgraphKK \}}$, we have ${\ptw{\alpha}{Q_n(\mathcal{X}_n)} \leq 2}$ and~${\local{\alpha}{Q_n(\mathcal{X}_n)} \leq 3}$. 
\end{lemma}

\begin{proof}
    For each positive integers~$m$ and $m'$, let~${P_{m}}$ denote the path with~$m$ vertices~${t_1, \dots t_{m}}$ and edges~$t_i t_{i+1}$ for each~${i \in [m-1]}$, and let~$T_{m,m'}$ denote the tree obtained from~$P_m$ by, for each~${i \in m}$, attaching~$m'$ leaves~${t_i^1, \dots, t_i^{m'}}$ to~$t_i$. 
    
    For~${\mathcal{X} \in \{ \matchingKK, \antimatchingKK, \halfgraphKK \}}$, we do the following. 
    We define a path-decomposition~$(P_{n-1}, \beta)$ of~$Q_n(\mathcal{X}_n)$ by, for all~${i \in [n-1]}$, setting ${\beta(t_i) \coloneqq X^i \cup X^{i+1} \cup \{y^i, y^{i+1}, z^i\}}$ if~${\mathcal{X} =\antimatchingKK}$, and setting ${\beta(t_i) \coloneqq X^i \cup X^{i+1} \cup \{y^i, y^{i+1}\}}$, otherwise. 
    It is easy to see that~$(P_{n-1},\beta)$ is indeed a path-decomposition of~${Q_n(\mathcal{X}_n)}$. 
    And since each~$X^i \cup \{y^i\}$ induces a clique of~${Q_n(\mathcal{X}_n)}$, the independence number of each bag is at most~$2$. 
    (Indeed, $Q_n(\matchingKK_n)$ and $Q_n(\antimatchingKK_n)$ contain $C_4$ as an induced subgraph, hence are not chordal, and therefore have $\alpha$-treewidth equal to~$2$, and $Q_n(\halfgraphKK_n)$ is chordal and hence has~$\alpha$-treewidth equal to~$1$.) 
    
    We define a tree-decomposition~$(T_{\lfloor n/2 \rfloor, n}, \beta)$ of~$Q_n(\matchingKI_n)$ by setting ${\beta(t_i) = X^i \cup X^{i+2} \cup \{ y^i, y^{i+2} \}}$ for all positive integers~${i}$ with~${2i - 1 \in [n-1]}$, and setting~$\beta(t_i^j) = N_{Q_n(\matchingKI_n)}[x_j^{i+1}]$ for all positive integers~${i}$ with~${2i - 1 \in [n-1]}$ and all integers~${j \in [n]}$. 
    Again, it is easy to see that~$(T_{\lfloor n/2 \rfloor, n}, \beta)$ is a tree-decomposition of independence number~$2$. 
    (And since $Q_n(\matchingKI_n)$ contains~$C_6$ as an induced subgraph, it is not chordal and hence has $\alpha$-treewidth equal to~$2$.)

    Lastly, it is straight forward to observe that the neighbourhood of each vertex can be covered by at most three cliques, hence~${\local{\alpha}{Q_n(\mathcal{X}_n)} \leq 3}$ for each~${\mathcal{X}_n \in \{\matchingKI_n, \matchingKK_n, \antimatchingKK_n, \halfgraphKK_n \}}$. 
\end{proof}

To analyse the rankwidth of these graphs, we make use of the grid theorem for vertex-minors~\cite{GeelenKMW2023-vertexminorgridtheorem}. 
To state this, we first need to define vertex-minors. 

Let~$G$ be a graph and let~${v \in V(G)}$. 
We define the operation~${G \ast v}$ of \emph{local complementation at~$v$} as the graph where we complement the subgraph induced by the open neighbourhood of~$v$. 
A \emph{vertex-minor} of~$G$ is a graph obtained from~$G$ by a sequence of vertex-deletions and local completions. 
It is well-known that rankwidth is monotone under the vertex-minor containment relation~\cite{Oum2005-rw-vm}. 

\begin{lemma}[{\cite[Lemma 7.8]{GeelenKMW2023-vertexminorgridtheorem}}]
    \label{lem:pathconstellations}
    There exists a function~$f$ such that for any positive integer~$n$, 
    the graphs~$Q_{f(n)}(\matchingII_{f(n)})$ and~$Q_{f(n)}(\halfgraphII_{f(n)})$ have rankwidth at least~$n$. 
\end{lemma}

What Geelen, Kwon, McCarty, and Wollan actually show is that these graphs contain the ${n \times n}$-comparability grid\footnote{The \emph{$n \times n$-comparability grid} is the graph with vertex set~$[n]^2$ where~$(i,j)$ and~$(i',j')$ are adjacent if and only if either~${i \leq i'}$ and~${j \leq j'}$ or~${i \geq i'}$ and~${j \geq j'}$. } as a vertex-minor.
This lemma is an essential tool to prove their celebrated Vertex-Minor Grid Theorem, which states that every graph with sufficiently large rankwidth has an ${n \times n}$-comparability grid as a vertex-minor. 
We do not need the full strength of this fundamental result of the structural theory of vertex-minors, but \cref{lem:pathconstellations} suffices for our purposes.
Using this lemma, we are able to argue that each of~$Q_n(\matchingKI_n)$, $Q_n(\matchingKK_n)$, $Q_n(\antimatchingKK_n)$, and $Q_n(\halfgraphKK_n)$ have large rankwidth as we are able to find $Q_{n'}(\matchingII_{n'})$ or~$Q_{n'}(\halfgraphII_{n'})$ for some sufficiently large~$n'$ as a vertex-minor. 

\begin{lemma}
    \label{lem:Q_n(X)-2}
    There exists a function~$f$ such that for every positive integer~${n}$ and each~${\mathcal{X} \in \{\matchingKI, \matchingKK, \antimatchingKK, \halfgraphKK \}}$, we have ${\rankwidth{Q_{f(n)}(\mathcal{X}_{f(n)})} \geq n}$. 
\end{lemma}

\begin{proof}
    For each integer~${m \geq 3}$, observe that the graph obtained from~$Q_{m}(\matchingKK_m)$ by successively locally complementing at~$y^1, \dots, y^m$ is equal to~$Q_{m}(\matchingII_m)$. 
    Similarly, for each integer~${m \geq 3}$, that the graph obtained from~$Q_{m}(\halfgraphKK_m)$ by successively locally complementing at~${y^1, \dots, y^m}$ is equal to~$Q_{m}(\halfgraphII_m)$. 
    For each integer~${m \geq 3}$, consider the graph obtained from~$Q_{m}(\matchingKI_m)$ by successively locally complementing at~$y^1, y^3, \dots$. 
    Now, we locally complement at each vertex of degree~$2$ and delete it afterwards; this corresponds to suppressing each degree-$2$ vertex. 
    Lastly, we delete all degree-$1$ vertices. 
    The resulting graph is isomorphic to~$Q_{\lfloor m/2 \rfloor - 1}(\matchingII_m)$. 

    Finally, for each integer~${m \geq 3}$, observe that the graph obtained from~$Q_{m}(\antimatchingKK_m)$ by successively locally complementing at~$z^1, \dots, z^{m-1}$ and then deleting~$\{z^1, \dots, z^{m-1}\}$ is equal to~$Q_{m}(\matchingII_m)$. 
\end{proof}

Finally, we can conclude with the proof of \cref{thm:rw-vs-alpha-tw2}. 

\begin{proof}[Proof of \cref{thm:rw-vs-alpha-tw2}.]
    Let~$\mathcal{G}$ denote the class of $\mathcal{F}$-free graphs. 
    
    First suppose that~$\mathcal{F}$ contains a graph from each~${\mathcal{X} \in \mathfrak{B}}$. 
    Then the parameter $\maxP{\local{\alpha}{},\reduced{\omega}{}}{}$ is bounded on~$\mathcal{G}$ by \cref{cor:parameterequivalence2}. 
    Hence, by \cref{thm:rw-vs-alpha-tw3}, $\rankwidth{}$ and $\ptw{\alpha}{}$ are equivalent on~$\mathcal{G}$, and since~$\local{\alpha}{}$ is bounded on~$\mathcal{G}$, we have that~$\maxP{\local{\alpha}{}, \rankwidth{}}{}$ and~$\ptw{\alpha}{}$ are equivalent on~$\mathcal{G}$. 
    
    Now suppose that~${P_4 \cong \matchingKI_2 \in \mathcal{F}}$ and~$\Star_n \in \mathcal{F}$. 
    Then, in particular, each graph in~$\mathcal{G}$ is a cograph. 
    Now, cographs have rankwidth at most~$1$ \cite{Oum2005-rw-vm} and the $\alpha$-treewidth of~$\Star_n$-free cographs is less than~$n$ \cite[Theorem 1.7]{DallardKKMMW2024}. 
    Hence, both~$\maxP{\local{\alpha}{}, \rankwidth{}}{}$ and~$\ptw{\alpha}{}$ are bounded on~$\mathcal{G}$, and hence are equivalent.  

    Now suppose that~$\mathcal{F}$ does not contain~$\mathcal{X}_n$ for some~${\mathcal{X} \in \mathfrak{B}}$ and all~${n \in \N}$, and does not contain~${\matchingKI_2 \cong P_4}$. 
    Hence, $\mathcal{X}_n \in \mathcal{G}$. 
    If~${\mathcal{X} = \Star}$, then note that since~$\ptw{\alpha}{\Star_n} = 1$ and~${\local{\alpha}{\Star_n} = n}$, we conclude that~$\maxP{\local{\alpha}{},\rankwidth{}}{}$ and~$\ptw{\alpha}{}$ cannot be equivalent on~$\mathcal{G}$. 
    If, on the other hand, $\mathcal{X} \in \mathfrak{B} \setminus \{\Star\}$, 
    then~$\mathcal{F}$ contains only graphs of the form~$\mathcal{Y}_m$ for~$\mathcal{Y} \in \mathfrak{B} \setminus \{\mathcal{X}\}$ and~${m \in \N}$. 
    Since~$\matchingKI_1 = \matchingKK_1 = \antimatchingKK_1 = \halfgraphKK_1$, none of these is contained in~$\mathcal{F}$. 
    Moreover, since~$\matchingKK_2 = \antimatchingKK_2$, if~${\mathcal{X} \in \{ \matchingKK, \antimatchingKK \}}$, then~${\matchingKK_2 = \antimatchingKK_2 \notin \mathcal{F}}$. 
    Now, for each~$\mathcal{Y} \in \mathfrak{B} \setminus \{\mathcal{X}\}$ it is easy to see that~$Q_n(\mathcal{X}_n)$ does not contain~$\mathcal{Y}_m$ as an induced subgraph for any~${m \geq 3}$. 
    And since~${\matchingKI_2 \notin \mathcal{F}}$, it then follows that~$Q_n(\mathcal{X}_n) \in \mathcal{G}$ for every~${n \in \N}$. 
    Now since, by \cref{lem:Q_n(X)-2}, the rankwidth of these graphs is unbounded but, by \cref{lem:Q_n(X)-1}, the $\alpha$-treewidth (and local independence number) is bounded, we conclude that~$\maxP{\local{\alpha}{},\rankwidth{}}{}$ and~$\ptw{\alpha}{}$ cannot be equivalent on~$\mathcal{G}$.
\end{proof}

\bibliographystyle{abbrv}
\bibliography{literature}

\end{document}